\documentclass[12pt]{amsart}
\usepackage{amsmath}
\usepackage{amsthm}
\usepackage{amssymb}
\usepackage{amscd}
\usepackage{amsfonts}
\usepackage{amsbsy}
\usepackage{epsfig,afterpage}
\usepackage[dvips]{psfrag}
\usepackage{subfigure}
\usepackage{epsfig}
\usepackage{amsmath}
\usepackage{srcltx}
\usepackage{amsfonts}
\usepackage{amssymb}
\usepackage{psfrag}
\usepackage{verbatim}
\usepackage{graphicx}
\usepackage{amsmath}
\usepackage{amsthm}
\usepackage{amssymb}
\usepackage{amscd}
\usepackage{amsfonts}
\usepackage{amsbsy}
\usepackage{epsfig}
\usepackage[dvips]{psfrag}
\usepackage{multirow}
\usepackage{graphics}
\usepackage{epstopdf}
\usepackage{overpic}
\usepackage{float}
\usepackage{multicol}

\usepackage{graphicx}

\usepackage[colorlinks=true, linkcolor=blue, citecolor=red, urlcolor=blue]{hyperref}

\newtheorem {theorem} {Theorem}
\newtheorem {proposition} [theorem]{Proposition}
\newtheorem {corollary} [theorem]{Corollary}
\newtheorem {lemma}  [theorem]{Lemma}
\newtheorem {example} [theorem]{Example}
\newtheorem {remark} [theorem]{Remark}

\newtheorem {definition} [theorem]{Definition}

\newtheorem{mtheorem}{Theorem}

\usepackage{float}
\usepackage{tikz, wrapfig}
\tikzset{node distance=3cm, auto}
\allowdisplaybreaks

\begin{document}

\title[Slow-fast systems on the Poincar\'e--Lyapunov sphere]
{Polynomial slow-fast systems on the Poincar\'e--Lyapunov sphere}

\author[O. H. Perez and P. R. da Silva]
{Otavio Henrique Perez$^{1}$ and Paulo Ricardo da Silva$^{2}$}

\address{$^{1}$University of S\~{a}o Paulo (USP), Institute of Mathematics and Computer Science. Avenida Trabalhador S\~{a}o Carlense, 400, CEP 13566-590, S\~{a}o Carlos, S\~{a}o Paulo, Brazil.}

\address{$^{2}$S\~{a}o Paulo State University (UNESP), Institute of Biosciences, Humanities and
	Exact Sciences. Rua C. Colombo, 2265, CEP 15054--000. S. J. Rio Preto, S\~ao Paulo,
	Brazil.}

\email{otavio.perez@icmc.usp.br}
\email{paulo.r.silva@unesp.br}

\thanks{ .}


\date{}
\dedicatory{}
\maketitle

\begin{abstract}
The main goal of this paper is to study compactifications of polynomial slow-fast systems. More precisely, the aim is to give conditions in order to guarantee normal hyperbolicity at infinity of the Poincaré--Lyapunov sphere for slow-fast systems defined in $\mathbb{R}^{n}$. For the planar case, we prove a global version of the Fenichel Theorem, which assures the persistence of invariant manifolds in the whole Poincaré--Lyapunov disk. We also discuss the appearence of non normally hyperbolic points at infinity, namely: fold, transcritical and pitchfork singularities.
\end{abstract}

\section{Introduction}

Slow-fast systems are well-known in the literature due to their vast importance in applied sciences. For instance, the \emph{van der Pol system} \cite{VdP} was introduced in order to study a vacuum tube triode circuit. Applications in biology can be found in \cite{Hek} and in the references therein, and we refer the book \cite{Kuehn} for applications in other branches of sciences.

A system of ODEs of the form
\begin{equation}\label{eq-def-slowfast-1}
\varepsilon\dot{x} = P(x,\mathbf{y},\varepsilon); \ \quad \ \dot{\mathbf{y}} = \mathbf{Q}(x,\mathbf{y},\varepsilon);
\end{equation}
is called \emph{slow-fast system}, where $x\in\mathbb{R}$, $\mathbf{y} = (y_{2},\dots,y_{n})\in\mathbb{R}^{n-1}$, $0 < \varepsilon \ll 1$. For our purposes, it will supposed that
$$P:\mathbb{R}^{n+1}\rightarrow\mathbb{R}; \ \quad \ \mathbf{Q} = (Q^{2},\dots,Q^{n}):\mathbb{R}^{n+1}\rightarrow\mathbb{R}^{n-1}$$
are polynomial functions with respect to the $x,\mathbf{y}$ variables and analytic with respect to $\varepsilon$. Throughout this paper, $x$ and $\mathbf{y}$ will be called \emph{fast} and \emph{slow variables}, respectively. Setting $\varepsilon = 0$ in equation \eqref{eq-def-slowfast-1}, we obtain the so called \emph{slow system} given by
\begin{equation}\label{eq-def-slow-system}
    0 = P(x,\mathbf{y},0), \ \quad \ \dot{\mathbf{y}} = \mathbf{Q}(x,\mathbf{y},0);
\end{equation}
which is not an ODE, but it is an \textit{algebraic differential equation} (ADE). Solutions of \eqref{eq-def-slow-system} are contained in the \emph{affine algebraic variety}
$$C_{0} = \Big{\{}(x,\mathbf{y})\in\mathbb{R}\times\mathbb{R}^{n-1}; \quad P(x,\mathbf{y},0) = 0\Big{\}};$$
which will be called \emph{critical set} or \emph{critical manifold}. Throughout this paper, $C_{0}$ is a codimension one affine algebraic variety.

In equation \eqref{eq-def-slowfast-1}, the dot $\cdot$ represents the derivative of the functions $x(\tau)$ and $\mathbf{y}(\tau)$ with respect to the variable $\tau$. By taking $\varepsilon t = \tau$, the system \eqref{eq-def-slowfast-1} can be written as
\begin{equation}\label{eq-def-slowfast-2}
x' = P(x,\mathbf{y},\varepsilon); \ \quad \ \mathbf{y}' = \varepsilon \mathbf{Q}(x,\mathbf{y},\varepsilon).
\end{equation}

The apostrophe ' in \eqref{eq-def-slowfast-2} denotes the derivative of $x(t)$ and $\mathbf{y}(t)$ with respect to the variable $t$. Setting $\varepsilon = 0$ in equation \eqref{eq-def-slowfast-2} we obtain
\begin{equation}\label{eq-def-fast-system}
x' = P(x,\mathbf{y},0); \ \quad \ \mathbf{y}' = 0.
\end{equation}
which will be called \emph{fast system}. The system \eqref{eq-def-fast-system} can be seen as a system of ordinary differential equations, with $\mathbf{y}\in\mathbb{R}^{n}$ being a parameter and the critical set $C_{0}$ is a set of equilibrium points of \eqref{eq-def-fast-system}.

Observe that the systems \eqref{eq-def-slowfast-1} and \eqref{eq-def-slowfast-2} are equivalent if $\varepsilon > 0$, since they differ by time scale.  The main challenge is to study systems \eqref{eq-def-slow-system} and \eqref{eq-def-fast-system} in order to obtain information of the full system \eqref{eq-def-slowfast-1}. For this purpose, the key tool that will be used in this paper is \emph{Geometric Singular Perturbation Theory} (GSPT for short). Neil Fenichel's seminal work \cite{Fenichel} assures that, under the hypothesis of \emph{normal hyperbolicity}, compact limit sets persist for small perturbations. See Subsection \ref{subsec-fenichel} for further details.

This paper is dedicated to study conditions in order to assure normal hyperbolicity near infinity. This problem was motivated by \cite{SarmientoOliveiraSilva}, in which all possible global phase portraits of quadratic slow-fast systems defined in the plane were given. In such reference, the authors conjectured a \emph{global version} of Fenichel Theorem for quadratic planar slow-fast systems. The contribution of the present paper is to give an answer of this problem for polynomial slow-fast systems in general.

The Poincar\'e compactification is a well known approach used in the study of global dynamics of polynomial vector fields. The main ideas were introduced in \cite{Poincare} by Henri Poincar\'e for the $2$-dimensional case. We refer to \cite{CimaLlibre, DLA, GV, Perko} for details of such technique, including the case where the polynomial vector field is defined in $\mathbb{R}^{n}$.

In this study, we consider the so called \emph{Poincaré--Lyapunov compactification} (PL-compactification for short) of polynomial vector fields defined in $\mathbb{R}^{n}$. The PL-compactification can be seen as a generalization of the well known Poincaré compactification technique. The construction of the PL-compactification is very similar to the construction of the classical Poincaré compactification, in the sense that we make it quasi-homogeneous instead of homogeneous (see \cite{LimaLlibre}).

Such technique was utilized in several papers by Freddy Dumortier, for example in the study of \emph{Liénard equations} near infinity (see for instance \cite{Dumortier, DumortierHerssens, DumortierLi, DumortierRousseau}). In \cite{LiangHuangZhao, QiuLiang} was given all possible phase portraits in the \emph{Poincaré--Lyapunov disk} (PL-disk for short) of polynomial vector fields having isolated singularities with degree of quasihomogeneity $4$ and $5$, respectively. Structural stability of quasi homogeneous polynomial vector fields in the PL-disk was studied in \cite{OliveiraZhao}. Global dynamics of the \emph{Benoît system} (which is three dimensional) in the \emph{Poincaré--Lyapunov ball} (PL-ball for short) was considered in \cite{LimaLlibre}. We refer to \cite[Chapters 5 and 9]{DLA} for an introduction on such method.

Let $Y$ be a polynomial vector field and let $\omega = (\omega_{1},\dots,\omega_{n})\in\mathbb{Z}^{n}$ be a vector of positive integers, which will be called \emph{weight vector}. The \emph{Poincaré--Lyapunov compactification} $Y^{\infty}$ of $Y$ is an analytic vector field defined in a compact $n$-dimensional manifold called \emph{Poincaré--Lyapunov sphere} (PL-sphere), which is denoted by $\mathbb{S}^{n}_{\omega}$ and it is homeomorphic to $\mathbb{S}^{n} = \{\sum_{i = 1}^{n+1} z_{i}^{2} = 1\}\subset\mathbb{R}^{n+1}$. The phase space $\mathbb{R}^{n}$ is identified with the northern hemisphere of $\mathbb{S}^{n}_{\omega}$, and the set $\{z_{n+1} = 0\}\subset\mathbb{S}^{n}_{\omega}$ plays the role of the infinity. See subsection \ref{subsec-pl-compact} for details.

In the study of global dynamics of polynomial vector fields, the PL-compactification $Y^{\infty}$ is a vector field defined in $\mathbb{S}^{n}_{\omega}$, however, the global phase portrait is often sketched in the PL-ball. Throughout this paper, the $n$-dimensional PL-ball will be denoted by $\mathbb{B}^{n}_{\omega}$. In particular, the PL-disk will be denoted by $\mathbb{D}_{\omega} = \mathbb{B}^{2}_{\omega}$. The interior of $\mathbb{D}_{\omega}$ plays the role of the $\mathbb{R}^{2}$, while its boundary plays the role of the infinity. Analogously, the interior of the 3-dimensional PL-ball $\mathbb{B}^{3}_{\omega}$ plays the role of $\mathbb{R}^{3}$ and its boundary represents the infinity. See Figure \ref{fig-pl-disk-ball}.

\begin{figure}[h!]
  \begin{overpic}[scale=1.3]{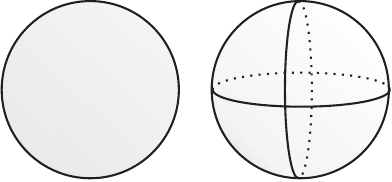}
		\put(21,-5){{$\mathbb{D}_{\omega}$}}
		\put(74,-5){{$\mathbb{B}^{3}_{\omega}$}}
		\end{overpic}
  \caption{\footnotesize{PL-disk (left) and 3-dimensional PL-ball (right).}}
  \label{fig-pl-disk-ball}
\end{figure}

Many interesting phenomena can occur at infinity of the phase space. For instance, consider the slow-fast system
\begin{equation}\label{eq-vdp-r3}
x' = y^{2}z - \frac{x^{2}y}{2} -   \frac{x^{3}}{3}; \quad y' = 0; \quad z' = \varepsilon (ay^{3} - xy^{2}). 
\end{equation}

After Poincaré compactification, in one of the three charts the following system is obtained
\begin{equation}\label{eq-vdp-infinity}
u' = v - \frac{u^{2}}{2} - \frac{u^{3}}{3}; \quad v' = \varepsilon (a-x);   
\end{equation}
which is the van der Pol system studied in \cite{DumRou}. See also Figure \ref{fig-vanderpol}. More generally, the system
\begin{equation}\label{eq-vdp-pl-r3}
x' = y^{k_{1}}z - \frac{x^{2}y^{k_{2}}}{2} -   \frac{x^{3}y^{k_{3}}}{3}; \quad y' = 0; \quad z' = \varepsilon (ay^{k_{4}} - xy^{k_{5}})
\end{equation}
presents a van der Pol system at infinity after a PL-compactification with weights $\omega = (\omega_{1},\omega_{2},\omega_{3})$ if, and only if, the positive integers $k_{1},\dots,k_{5}$ satisfy\noindent\begin{multicols}{3}
$$k_{1} = \frac{\delta + \omega_{1} - \omega_{3}}{\omega_{2}};$$
$$k_{2} = \frac{\delta - \omega_{1}}{\omega_{2}};$$
$$k_{3} = \frac{\delta - 2\omega_{1}}{\omega_{2}};$$
$$k_{4} = \frac{\delta + \omega_{3}}{\omega_{2}};$$
$$k_{5} = \frac{\delta - \omega_{1} + \omega_{3}}{\omega_{2}}.$$
\end{multicols}

\begin{figure}[h!]
  \center{\includegraphics[width=0.40\textwidth]{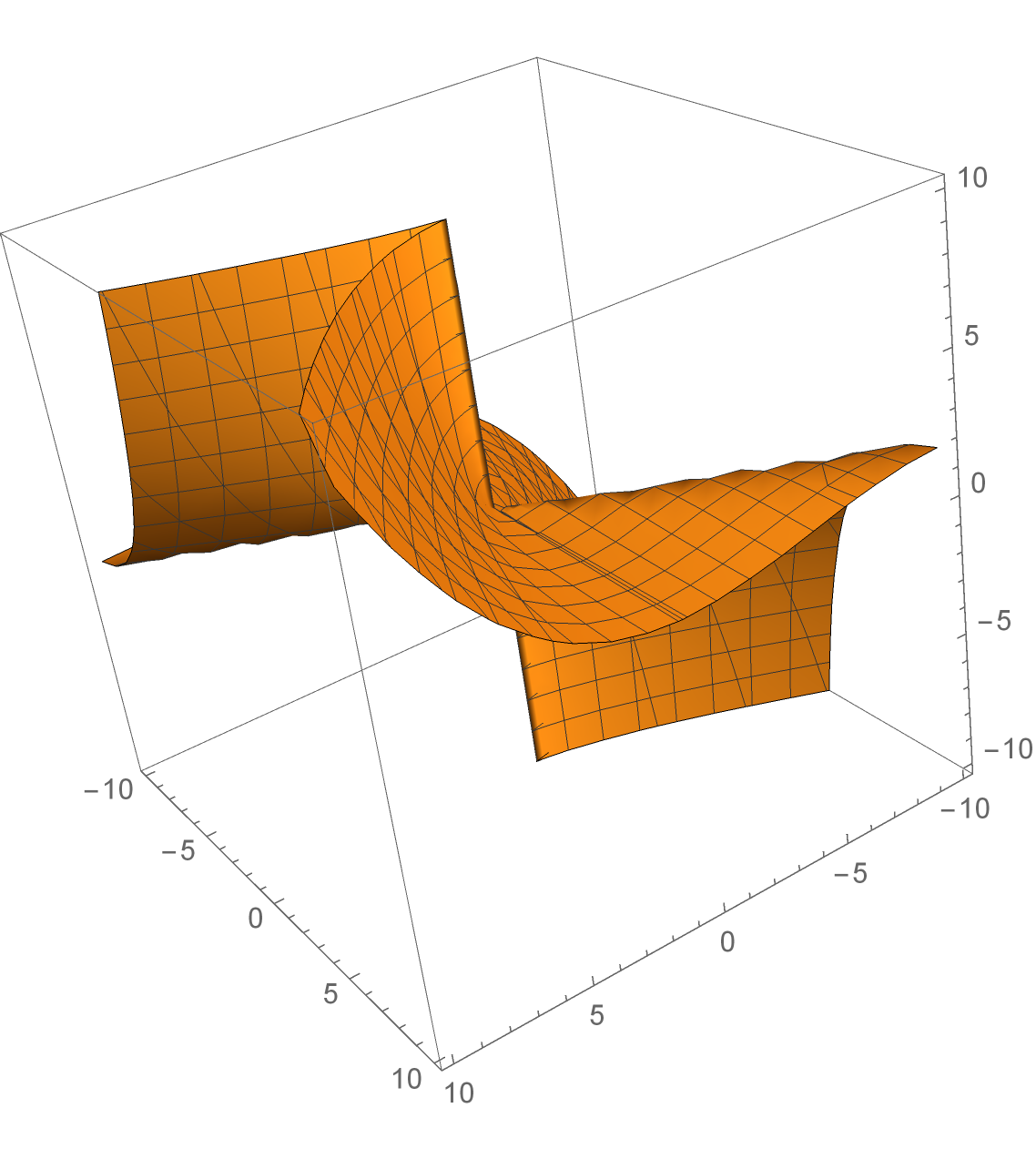}\hspace{0.55cm}\includegraphics[width=0.35\textwidth]{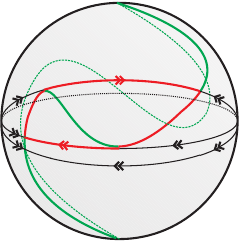}}\\
  \caption{\footnotesize{Critical manifold of slow-fast system \eqref{eq-vdp-r3} (left) and its phase portrait at infinity (right), which is given by the van der Pol equation \eqref{eq-vdp-infinity}. The critical manifold is highlithed in green and the canard cycle is highlighted in red (see also \cite{DumRou}).}}
  \label{fig-vanderpol}
\end{figure}

Let us briefly describe our main results. A preliminary and useful result is given in Proposition \ref{prop-general-dim}, which discusses the possible dynamics at infinity of a compactified slow-fast system based on the degree of quasi homogeneity $P$ and $\mathbf{Q}$. Afterwards, in Theorem \ref{teo-fenichel-infty} we state Fenichel Theorem in a suitable way in order to study the perturbed system at infinity (boundary of the PL-ball).

In Theorem \ref{teo-a} we state conditions that polynomial slow-fast systems in $\mathbb{R}^{n}$ must satisfy in order to assure normal hyperbolicity at infinity (boundary of the PL-ball). More precisely, item (a) of Theorem \ref{teo-a} states an algebraic condition on the polynomial $P$ which the initial slow-fast system must satisfy so that the origin of each chart of the PL-ball is normally hyperbolic. Such condition implies that, on the \emph{Newton polytope} of the slow-fast vector field (see Subsection \ref{subsec-newton-polytope} for a precise definition), the points associated to higher order monomials are all contained in the same $(n-1)$-dimensional compact face of the polytope. We emphasize that Theorem \ref{teo-a} item (a) concerns the origin of each chart of the compactification.

On the other hand, Theorem \ref{teo-a} item (c) we give a \emph{necessary} condition in order to assure normal hyperbolicity outside the origin, and such condition is based on the transversal intersection of the critical manifold with the infinity. Finally, Theorem \ref{teo-a} item (b) concerns a degenerate case, in which the \emph{whole} infinity is a component of the critical manifold.

In dimension 2, Theorem \ref{teo-fenichel-global-r2} gives sufficient and necessary conditions in order to assure the persistence of invariant manifolds in the \emph{whole} PL-disk. Actually, transversality turns out to be a necessary and sufficient condition in order to assure normal hyperbolicity at infinity. Finally, Theorem \ref{teo-normal-forms} determine conditions that slow-fast systems defined in $\mathbb{R}^{3}$ must satisfy in order to generate typical singularities of planar slow-fast systems at infinity, namely fold, transcritical and pitchfork singularities.

This paper is structured as follows. In Section \ref{sec-preliminaires} is presented some preliminaries on GSPT and Poincaré--Lyapunov compactification. Section \ref{sec-intermediate} is devoted to discuss some preliminary propositions and examples that will be used in the subsequent sections. Theorem \ref{teo-a} is proven in Section \ref{sec-NH-RN}, and Section \ref{sec-planar-gspt} is devoted to give the proof of the global version of Fenichel Theorem in the plane (Theorem \ref{teo-fenichel-global-r2}). Finally, in Section \ref{sec-non-NH} is proven Theorem \ref{teo-normal-forms} and some examples are also given.

\section{Preliminaries on geometric singular perturbation theory and Poincar\'e--Lyapunov compactification}\label{sec-preliminaires}

\subsection{Geometric singular perturbation theory}\label{subsec-fenichel}

A point $p\in C_{0}$ is \emph{normally hyperbolic} if $P_{x}(p) \neq 0$. The set of all normally hyperbolic points of $C_{0}$ will be denoted by $\mathcal{NH}(C_{0})$. A point $p\in\mathcal{NH}(C_{0})$ is called \emph{attracting point} if $P_{x}(p) < 0$; and it is called \emph{repelling point} if $P_{x}(p) > 0$. 

The \emph{Fenichel Theorem} is a major result in Geometric Singular Perturbation Theory. It assures that, given a $j$-dimensional compact normally hyperbolic sub-manifold $\mathcal{K}\subset\mathcal{NH}(C_{0})$ (possibly with boundary) of the slow system \eqref{eq-def-slow-system}, there exists a family of smooth manifolds $\mathcal{K}_{\varepsilon}$ such that $\mathcal{K}_{\varepsilon} \rightarrow \mathcal{K}_{0} = \mathcal{K}$ according to Hausdorff distance and $\mathcal{K}_{\varepsilon}$ is a normally hyperbolic locally invariant manifold of \eqref{eq-def-slowfast-1}. Such result was first proved in \cite{Fenichel}. See also \cite[Theorem 2.2]{Szmolyan} for a precise statement. In Theorem \ref{teo-fenichel-infty}, we stated the Fenichel Theorem in a suitable way in order to assure the persistence of invariant manifolds at infinity.

The Fenichel Theorem can be seen as a ``generalization'' of the theorem of the stable and unstable manifolds. The \textit{local invariance} of $\mathcal{K}_{\varepsilon}$ means that it may exist boundaries through which trajectories can leave. Just as in center manifold theory, in general the locally invariant manifold $\mathcal{K}_{\varepsilon}$ obtained in the Fenichel Theorem is not unique. Indeed, it may exist infinitely many invariant manifolds $\mathcal{O}(e^{-\frac{K}{\varepsilon}})$-close to the critical manifold. See Figure \ref{fig-fenichel}. The manifold $\mathcal{K}_{\varepsilon}$ obtained in the Fenichel Theorem is called \emph{slow manifold}.

\begin{figure}[h!]
  \begin{overpic}[scale=1]{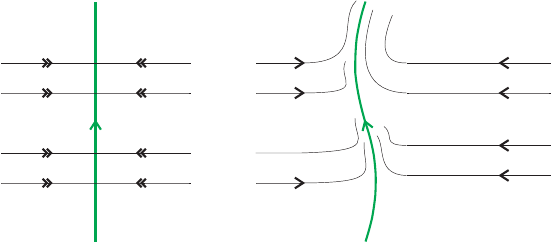}
		\put(10,0){{ $\mathcal{K}$}}
		\put(70,0){{ $\mathcal{K}_{\varepsilon}$}}
		\end{overpic}
  \caption{\footnotesize{Planar slow-fast system for $\varepsilon = 0$ (left) and for $\varepsilon > 0$ sufficiently small (right). The Fenichel Theorem assures the existence of a family of invariant manifolds $\mathcal{K}_{\varepsilon}$, and the flow on $\mathcal{K}_{\varepsilon}$ converges to the flow on $\mathcal{K}$. Moreover, Fenichel Theorem also assure the existence of a family of stable manifolds $\mathcal{W}^{s}_{\varepsilon}$ of $\mathcal{K}_{\varepsilon}$.}}
  \label{fig-fenichel}
\end{figure}

Concerning the slow system \eqref{eq-def-slow-system}, \emph{``any structure in $\mathcal{NH}(C_{0})$ which persists under regular perturbations persists under singular perturbation''} \cite[pp.~91]{Fenichel}. In other words, hyperbolic equilibrium points or limit cycles of \eqref{eq-def-slow-system} in $\mathcal{NH}(C_{0})$ persist for $\varepsilon$ sufficiently small.

The Fenichel Theorem gives an answer about the dynamics of system \eqref{eq-def-slowfast-1} near normally hyperbolic manifolds for $\varepsilon$ sufficiently small. We refer to \cite{DumRou, DumRou2, KrupaSzmolyan, KrupaSzmolyan2} for further problems and techniques concerning the dynamics of \eqref{eq-def-slowfast-1} near non-normally hyperbolic manifolds.

\subsection{Poincaré--Lyapunov compactification of polynomial vector fields}\label{subsec-pl-compact}

Just as in the homogeneous compactification, the vector field $Y^{\infty}$ is studied using directional charts $U_{i}$ and $V_{i}$, in which
$$
U_{i} = \{\mathbf{z}\in\mathbb{S}^{n}_{\omega}; \ z_{i} > 0\}; \  V_{i} = \{\mathbf{z}\in\mathbb{S}^{n}_{\omega}; \ z_{i} < 0\}; \  \mathbf{z} = (z_{1},\dots,z_{n+1})\in\mathbb{R}^{n+1};
$$
for each $i = 1,\dots, n+1$. See Figure \ref{fig-PL-compactification}.

Consider the polynomial vector field $Y(\mathbf{x}) = Y(x_{1},\dots,x_{n})$. For every $i = 1,\dots, n$, the expression of the compactified vector field $Y^{\infty}(\mathbf{u}) = Y^{\infty}(u_{1},\dots,u_{n})$ in the charts $U_{i}$ is obtained from the change of coordinates
$$x_{1} = \displaystyle\frac{u_{1}}{u_{n}^{\omega_{1}}}, \ \ldots, \ x_{i-1} = \displaystyle\frac{u_{i-1}}{u_{n}^{\omega_{i-1}}}, \ x_{i} = \displaystyle\frac{1}{u_{n}^{\omega_{i}}}, \ x_{i+1} = \displaystyle\frac{u_{i}}{u_{n}^{\omega_{i+1}}}, \ \ldots, \ x_{n} = \displaystyle\frac{u_{n-1}}{u_{n}^{\omega_{n}}};$$
and for different charts $U_{i}$ the coordinate system $(u_{1},\ldots,u_{n})$ has different meanings. However, for every $i = 1,\dots, n$ the set $\{u_{n} = 0\}$ is an invariant set of $Y^{\infty}$ which plays the role of the infinity. 

On the other hand, the expression of $Y^{\infty}(\mathbf{u}) = Y^{\infty}(u_{1},\dots,u_{n})$ in the charts $V_{i}$ is obtained in the same way as in the charts $U_{i}$, but setting $x_{i} = -\displaystyle\frac{1}{u_{n}^{\omega_{i}}}$ instead of $x_{i} = \displaystyle\frac{1}{u_{n}^{\omega_{i}}}$.

For $i = n+1$, the expression of $Y^{\infty}$ in $U_{n+1}$ coincides with the expression of the vector field $Y$ In the chart $V_{n+1}$, the expression of $Y^{\infty}$ coincides with the expression $Y$ (up to a multiplication by -1).

\section{Poincaré--Lyapunov compactification of slow-fast systems}\label{sec-intermediate}

Consider the polynomial slow-fast system \eqref{eq-def-slowfast-2}. Recall that $P$ and $\mathbf{Q}$ are polynomial with respect to the fast variable $x$ and the slow variables $\mathbf{y}$, but it is analytic with respect to $\varepsilon$.

In what follows we present the definitions of quasi homogeneous polynomial and quasi homogeneous vector field, which can also be found in \cite[Section 7.3]{Kuehn}. Let $\omega = (\omega_{1},\dots,\omega_{n})\in\mathbb{Z}^{n}$ be a weight vector. A polynomial $F:\mathbb{R}^{n}\rightarrow\mathbb{R}$ is \emph{quasi homogeneous of type $\omega$ and degree $k\in\mathbb{N}$} if
\begin{equation*}
F(\lambda^{\omega_{1}}x,\lambda^{\omega_{2}}y_{2},\dots,\lambda^{\omega_{n}}y_{n}) = \lambda^{k}\cdot F(x,y_{2},\dots,y_{n}); \ \quad \ \forall\lambda\in\mathbb{R}.
\end{equation*}

We say that a polynomial vector field $Y = (Y_{1},\dots,Y_{n})$ defined in $\mathbb{R}^{n}$ is \emph{quasi homogeneous of type $\omega$ and degree $k_{\omega}\in\mathbb{N}$} if each component $Y_{j}:\mathbb{R}^{n}\rightarrow\mathbb{R}$ of $Y$ is quasi homogeneous of type $\omega$ and degree $k+\omega_{j}$. In other words, it satisfies
\begin{equation*}
Y_{j}(\lambda^{\omega_{1}}x,\lambda^{\omega_{2}}y_{2},\dots,\lambda^{\omega_{n}}y_{n}) = \lambda^{k + \omega_{j}}\cdot Y_{j}(x,y_{2},\dots,y_{n}); \ \quad \ \forall\lambda\in\mathbb{R}.
\end{equation*}

\begin{example}\label{exe-planar-cusp}
Consider the planar polynomial vector field
$$Y(x,y) = \big{(}Y_{1}(x,y),Y_{2}(x,y)\big{)} = (y,x^{2})$$
which determines a cusp singularity at the origin. This vector field is quasi homogeneous of type $\omega = (2,3)$ and degree $1$, because
$$Y_{1}(\lambda^{2}x,\lambda^{3}y) = \lambda^{1 + 2}Y_{1}(x,y), \ \quad \ Y_{2}(\lambda^{2}x,\lambda^{3}y) = \lambda^{1 + 3}Y_{2}(x,y).$$
\end{example}

The vector field associated to the slow-fast system \eqref{eq-def-slowfast-2} will be denoted by $X_{\varepsilon}$, whereas its PL-compactification will be denoted by $X^{\infty}_{ \varepsilon}$, which is a vector field defined in $\mathbb{S}^{n}_{\omega}\subset\mathbb{R}^{n+1}$. We will also write the polynomial functions $P, Q^{j}$ as
$$P = \sum_{d=-1}^{\delta_{1}}P_{ d}, \ \quad \ Q^{j} = \sum_{d=-1}^{\delta_{j}}Q^{j}_{ d};$$
in which $P_{ d}$ is the quasi homogeneous component of type $\omega$ and degree $d+\omega_{1}$, and $Q^{j}_{ d}$ is the quasi homogeneous component of type $\omega$ and degree $d+\omega_{j}$. The \emph{degree of quasihomogeneity of type $\omega$ of $P_{ d}$ and $Q^{j}_{ d}$} will be denoted by
$\deg_{\omega}P_{ d}$ and $\deg_{\omega}Q^{j}_{ d}$. The \emph{highest degree of quasihomogeneity of type $\omega$ of $P$ and $Q^{j}$} is
$$
  \deg_{\omega}P = \displaystyle\max_{d}\{\deg_{\omega}P_{ d}\} = \delta_{1} + \omega_{1}, \ 
  \deg_{\omega}Q^{j} = \displaystyle\max_{d}\{\deg_{\omega}Q^{j}_{ d}\} = \delta_{j} + \omega_{j}.
$$

Then, the highest quasihomogeneous degree component of $P$ and $Q^{j}$ is, respectively, $P_{ \delta_{1}}$ and $Q^{j}_{ \delta_{j}}$. The degree of quasi homogeneity type $\omega$ of the vector field $X_{\varepsilon}$ will be simply denoted by $\deg_{\omega}X_{\varepsilon} = \max{\delta_{l}} = \delta$.

\begin{example}\label{exe-planar-cusp-2}
Let $Y$ be the planar polynomial vector field given in the Example \ref{exe-planar-cusp}. If $\omega = (2,3)$, then $\deg_{\omega}Y = 1$, $\deg_{\omega}Y_{1} = 1+2 = 3$ and $\deg_{\omega}Y_{2} = 1+3 = 4$. On the other hand, if we consider $\omega = (1,1)$, then $\deg_{\omega}Y = 1$, $\deg_{\omega}Y_{1} = 0+1 = 1$, and $\deg_{\omega}Y_{2} = 1+1 = 2$.
\end{example}

In what follows, the expressions of $X^{\infty}_{ \varepsilon}$ in each of the $2(n+1)$ local charts of $\mathbb{S}^{n}_{\omega}$ are given. The notation $(x,\mathbf{y},\varepsilon)$ concerns a coordinate system in the finite part of the phase space, whereas $(u,\mathbf{v},\varepsilon)$ concerns the coordinate system near infinity, in which $\mathbf{v} = (v_{2},\dots,v_{n})$. We emphasize that in different open sets $U_{i}$ of the covering, the coordinates $(u,\mathbf{v},\varepsilon)$ have different meanings. See Figure \ref{fig-PL-compactification}.

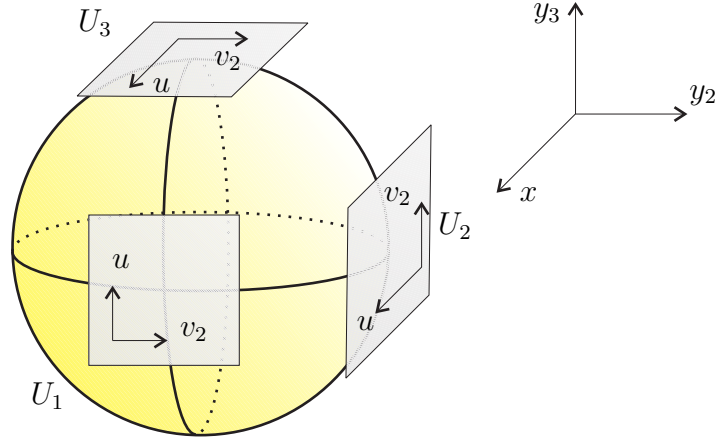
\begin{figure}[h!]
  \begin{overpic}[scale=1]{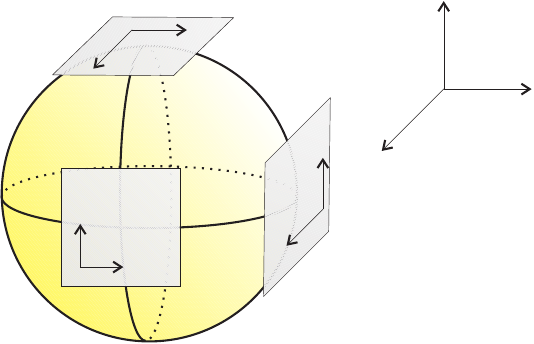}
		\put(3,5){$U_{1}$}
        \put(15,25){$u$}
        \put(25,15){$v_{2}$}
        \put(63,30){$U_{2}$}
        \put(51,16){$u$}
        \put(55,35){$v_{2}$}
        \put(10,60){$U_{3}$}
        \put(21,51){$u$}
        \put(30,55){$v_{2}$}
        \put(75,35){$x$}
        \put(100,50){$y_{2}$}
        \put(77,62){$y_{3}$}
		\end{overpic}
  \caption{\footnotesize{Directional charts that cover the PL-ball $\mathbb{B}^{3}_{\omega}$. Following the terminology of Definition \ref{def-directions}, the slow-fast vector field obtained in the chart $U_{1}$ is the compactification in the fast direction, whereas the vector field obtained in the chart $U_{l}$ is the compactification in the slow direction, for $l = 2,\dots,n$.}}
  \label{fig-PL-compactification}
\end{figure}

In $U_{1}$, the compactification is written as
\begin{equation}\label{eq-slow-fast-U1-n}
\left\{
  \begin{array}{rcl}
   u' & = & \displaystyle\sum_{d=-1}^{\delta}v_{n}^{\delta - d}\big{(}\varepsilon Q^{2}_{ d} - u\displaystyle\frac{\omega_{2}}{\omega_{1}}P_{ d}\big{)}, \\
   v'_{2} & = & \displaystyle\sum_{d=-1}^{\delta}v_{n}^{\delta - d}\big{(}\varepsilon Q^{3}_{ d} - v_{2}\displaystyle\frac{\omega_{3}}{\omega_{1}}P_{ d}\big{)}, \\
   \vdots & = & \vdots \\
    v'_{n-1} & = & \displaystyle\sum_{d=-1}^{\delta}v_{n}^{\delta - d}\big{(}\varepsilon Q^{n}_{ d} - v_{n-1}\displaystyle\frac{\omega_{n}}{\omega_{1}}P_{ d}\big{)}, \\
    v'_{n} & = & -\displaystyle\frac{1}{\omega_{1}}\displaystyle\sum_{d=-1}^{\delta}v_{n}^{\delta+1-d}P_{ d}; 
  \end{array}
\right.
\end{equation}
where $P_{ d}, Q^{j}_{ d}$ are computed in $(1,u,v_{2},\dots,v_{n-1},\varepsilon)$ for all $j = 2,\dots,n$. Observe that it was used the quasi homogeneity of the components $P_{ d}, Q^{j}_{ d}$ in order to obtain the System \eqref{eq-slow-fast-U1-n}.

For $l=2,\dots,n$, in $U_{l}$ the compactification is written as
\begin{equation}\label{eq-slow-fast-Ul-n}
\left\{
  \begin{array}{rclc}
        u' & = & \displaystyle\sum_{d=-1}^{\delta}v_{n}^{\delta - d}\big{(}P_{ d} - \varepsilon u\displaystyle\frac{\omega_{1}}{\omega_{l}}Q^{l}_{ d}\big{)}, & \\
        v'_{i} & = &  \varepsilon\displaystyle\sum_{d=-1}^{\delta}v_{n}^{\delta - d}\big{(}Q^{i}_{ d} - v_{i}\displaystyle\frac{\omega_{i}}{\omega_{l}}Q^{l}_{ d}\big{)}, & 1<i<l \\
        v'_{i-1} & = &  \varepsilon\displaystyle\sum_{d=-1}^{\delta}v_{n}^{\delta - d}\big{(}Q^{i}_{ d} - v_{i-1}\displaystyle\frac{\omega_{i}}{\omega_{l}}Q^{l}_{ d}\big{)}, & l<i\leq n \\
        v'_{n} & = & -\displaystyle\frac{\varepsilon}{\omega_{l}}\sum_{d=-1}^{\delta}v_{n}^{\delta+1-d}Q^{l}_{ d}, &
  \end{array}
\right.
\end{equation}
where $j = 2,\dots,n$, and the polynomial functions $P_{ d}, Q^{j}_{ d}$ are computed in  $(u, v_{2}, \dots, v_{l-1}, 1, v_{l}, \dots, v_{n-1},\varepsilon)$. Once again it was used the quasi homogeneity of $P_{ d}, Q^{j}_{ d}$ in order to obtain the System \eqref{eq-slow-fast-Ul-n}.

The expression of $X^{\infty}_{ \varepsilon}$ in $U_{n+1}$ is precisely the expression of the original vector field \eqref{eq-def-slowfast-2}. The expression of $X^{\infty}_{ \varepsilon}$ in the open set $V_{i}$ is obtained by replacing $\frac{1}{v_{n}^{\omega_{i}}}$ by $-\frac{1}{v_{n}^{\omega_{i}}}$ in the change of coordinates of the chart $U_{i}$, for all $i = 1,\dots,n+1$. Furthermore, in any local chart $U_{i}$ and $V_{i}$ the set $\{v_{n} = 0\}$ is an invariant set of $X^{\infty}_{ \varepsilon}$ that plays the role of the infinity.

\begin{definition}\label{def-directions}
The vector field obtained in the chart $U_{1}$ will be called \emph{compactification of $X_{\varepsilon}$ in the positive fast direction}. The vector field obtained in $U_{l}$, for $l=2,\dots,n$, will be called \emph{compactification of $X_{\varepsilon}$ in the $l$-th positive slow direction}.
\end{definition}

From the expressions of $X^{\infty}_{ \varepsilon}$ in the charts $U_{l}$ for $l = 1,\dots, n$, one can conclude the following proposition:

\begin{proposition}\label{prop-general-dim}
Let $X_{\varepsilon}$ be the polynomial vector field associated to the slow-fast system \eqref{eq-def-slowfast-2} and denote its PL-compactification by $X^{\infty}_{ \varepsilon}$. Then, in the charts $U_{l}$ for $l = 1,\dots, n$, it follows that:
\begin{description}
    \item[(a)] The PL-compactification of $X_{\varepsilon}$ in the fast direction is not a slow-fast system, but it is a singular perturbation problem. In addition, for $\varepsilon = 0$, the set of equilibria is given by $\{(0,0)\}\cup\{P_{ \delta_{1}}(1,u,v_{2},\dots,v_{n-1},0) = 0\}$.
    \item[(b)] The PL-compactification of $X_{\varepsilon}$ in the $l$-th slow direction is a slow-fast system, for all $l = 2,\dots, n$.
    \item[(c)] Suppose $X_{\varepsilon}$ is a $n$-dimensional vector field for $n \geq 3$. Then, for $l = 2,\dots, n$ , the vector field $X^{\infty}_{ \varepsilon}$ defines a slow-fast system at infinity $\{v_{n} = 0\}$ in the chart $U_{l}$ if, and only if, $\delta = \delta_{1} = \delta_{j_{0}}$, for some $j_{0}$.
    Moreover, such a slow-fast system has one fast variable and $n-2$ slow variables.
    \item[(d)] If $\delta = \delta_{1} > \delta_{j}$ for all $j$, then $\varepsilon\mathbf{Q}$ does not affect the dynamics at infinity $\{v_{n} = 0\}$. On the other hand, if $\delta_{1} < \delta_{j_{0}} = \delta$ for some $j_{0}$, in the limit $\varepsilon = 0$ the infinity is filled with equilibria. 
\end{description}
\end{proposition}
\begin{proof}
    Assertions (a) and (b) follow directly from the expressions \eqref{eq-slow-fast-U1-n} and \eqref{eq-slow-fast-Ul-n} of the vector fields in the fast and slow-directions, respectively. In order to prove assertions (c) and (d), assume that  $X_{\varepsilon}$ is a $n$-dimensional vector field for $n \geq 3$. Observe that, from Equation \eqref{eq-slow-fast-Ul-n}, if  $\delta_{1} > \delta_{j}$ for all $j$, then only terms of $P$ play role at infinity $\{v_{n} = 0\}$. In other words, if one sets $v_{n} = 0$ in Equation \eqref{eq-slow-fast-Ul-n}, then only terms of $P$ will remain, thus $\varepsilon\mathbf{Q}$ does not affect the dynamics at infinity. The same reasoning can be used to prove that, if  $\delta_{1} < \delta_{j_{0}}$ for some $j_{0}$, then only terms of $\varepsilon\mathbf{Q}$ play role at $\{v_{n} = 0\}$. Then, setting $\varepsilon = 0$, the infinity $\{v_{n} = 0\}$ is filled with equilibria. Finally, if $\delta = \delta_{1} = \delta_{j_{0}}$ for some $j_{0}$, then terms of both $P$ and $Q^{j_{0}}$ play role at infinity, and therefore dynamics at infinity of $U_{l}$ is given by a slow-fast system.   
\end{proof}

\begin{example}\label{exe-prop-planar}
Consider the planar slow-fast system
\begin{equation}\label{eq-exe-prop-planar}
 x' = P(x,y,\varepsilon) = -x; \ \quad \ y' = \varepsilon Q(x,y,\varepsilon) = \varepsilon(y^{2} - x^{3}).   
\end{equation}

After a PL-compactification with weights $\omega = (2,3)$, the systems obtained in the fast and slow directions $U_{1}$ and $U_{2}$ are, respectively, \begin{multicols}{2} \begin{equation}\label{eq-exe-infty-equilibria-1}\left\{\begin{array}{rcl}
   u' & = & \varepsilon (u^{2} - 1) + \frac{3uv^{3}}{2}; \\
   v' & = & \frac{v^{4}}{2};
  \end{array}
\right. \end{equation}

\begin{equation}\label{eq-exe-infty-equilibria}
\left\{
  \begin{array}{rcl}
   u' & = & \frac{2\varepsilon u(u^{3} - 1)}{3} - uv^{3}; \\
   v' & = & \frac{\varepsilon v(u^{3} - 1)}{3}.
  \end{array}
\right.
 \end{equation}
 \end{multicols}

In this example, we have $\delta_{1} < \delta_{2}$ because $\deg_{\omega}P = \delta_{1} + 2 = 2$ and $\deg_{\omega} Q = \delta_{2} + 3 = 6$. As expected from item (a) of Proposition \ref{prop-general-dim}, the system \eqref{eq-exe-infty-equilibria-1} defined in $U_{1}$ is not a slow-fast system. From item (d), the infinity is filled with equilibria when $\varepsilon = 0$ in equation \eqref{eq-exe-infty-equilibria}. Observe that in $U_{2}$ the critical manifold is given by $\{uv^{3} = 0\}$. See Figure \ref{fig-exe-prop-planar}.
\end{example}

\begin{example}
Consider the slow-fast system
\begin{equation}\label{exe-prop-r3-1}
 x' = P(x,y,z,\varepsilon) = x(y^{2} - z^{2}); \quad y' = \varepsilon; \quad z' = \varepsilon.
\end{equation}

After a PL-compactification with weights $\omega = (1,1,1)$ (which is the classical Poincaré compactification), the systems obtained in $U_{1}$, $U_{2}$ and $U_{3}$ are, respectively,
\begin{equation}
 u' = u(v^{2} - u^{2}) + \varepsilon w^{3}; \quad v' = v(v^{2} - u^{2}) + \varepsilon w^{3}; \quad w' = w(v^{2} - u^{2});   
\end{equation}

\begin{equation}
 u' = u(1 - v^{2} - \varepsilon w^{3}); \quad v' = \varepsilon w^{3}(1-v); \quad w' = -\varepsilon w^{4};   
\end{equation}

\begin{equation}
 u' = u(1 - v^{2} - \varepsilon w^{3}); \quad v' = \varepsilon w^{3}(1-v); \quad w' = -\varepsilon w^{4}.    
\end{equation}

As expected from item (d) of Proposition \ref{prop-general-dim}, it follows that, at infinity $\{w = 0\}$, only terms of $P$ play role. Moreover, from item (a), the compactification in the fast direction is not a slow-fast system.
\end{example}

\begin{figure}[h!]
  \begin{overpic}[scale=1.3]{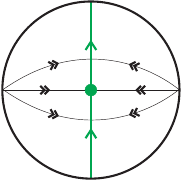}
	\end{overpic}
  \caption{\footnotesize{Phase portrait of the slow-fast system \eqref{eq-exe-prop-planar} in the PL-disk. The critical manifold is highlighted in green.}}
  \label{fig-exe-prop-planar}
\end{figure}

\begin{figure}[h!]
  \center{\includegraphics[width=0.35\textwidth]{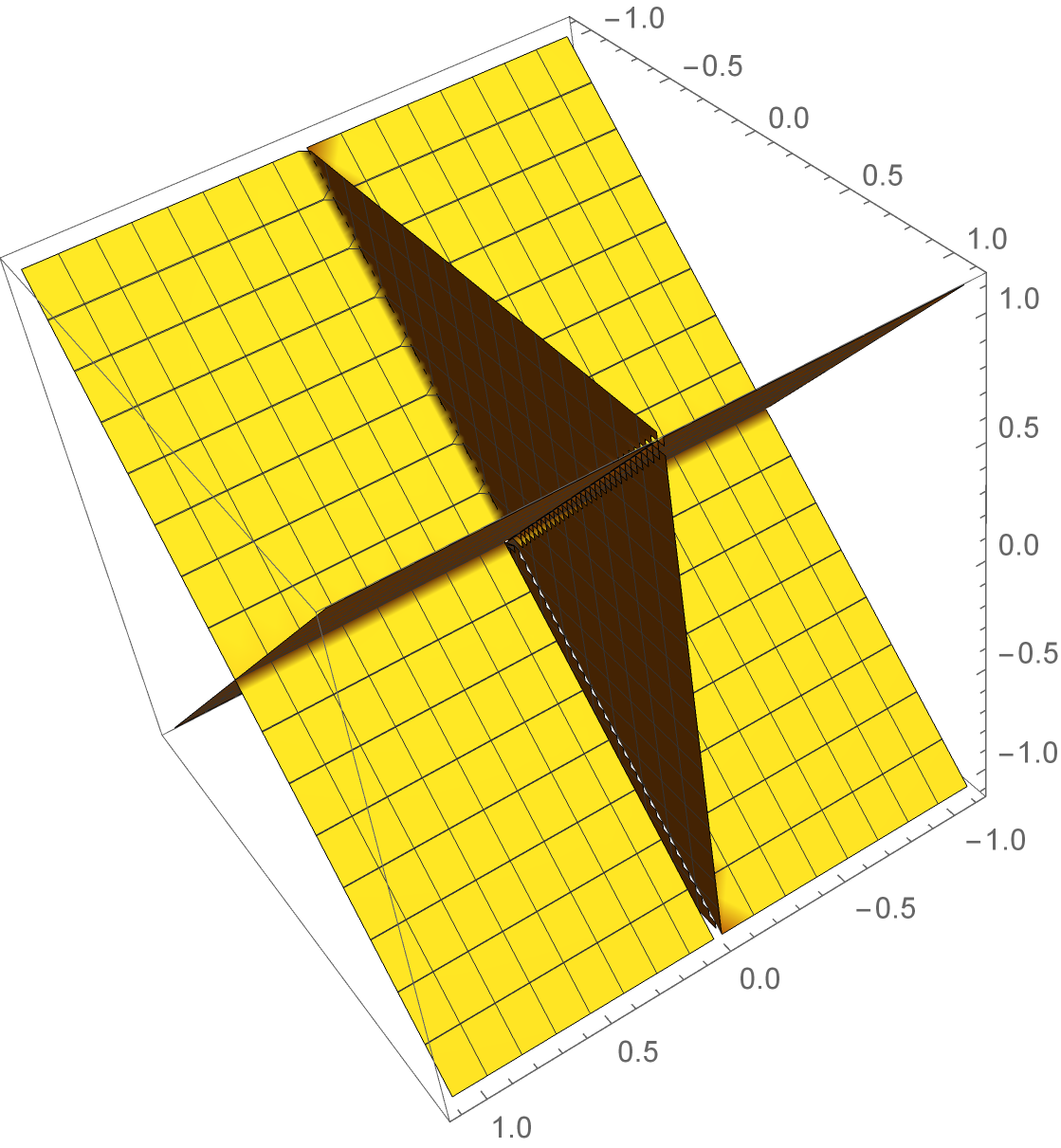}\includegraphics[width=0.35\textwidth]{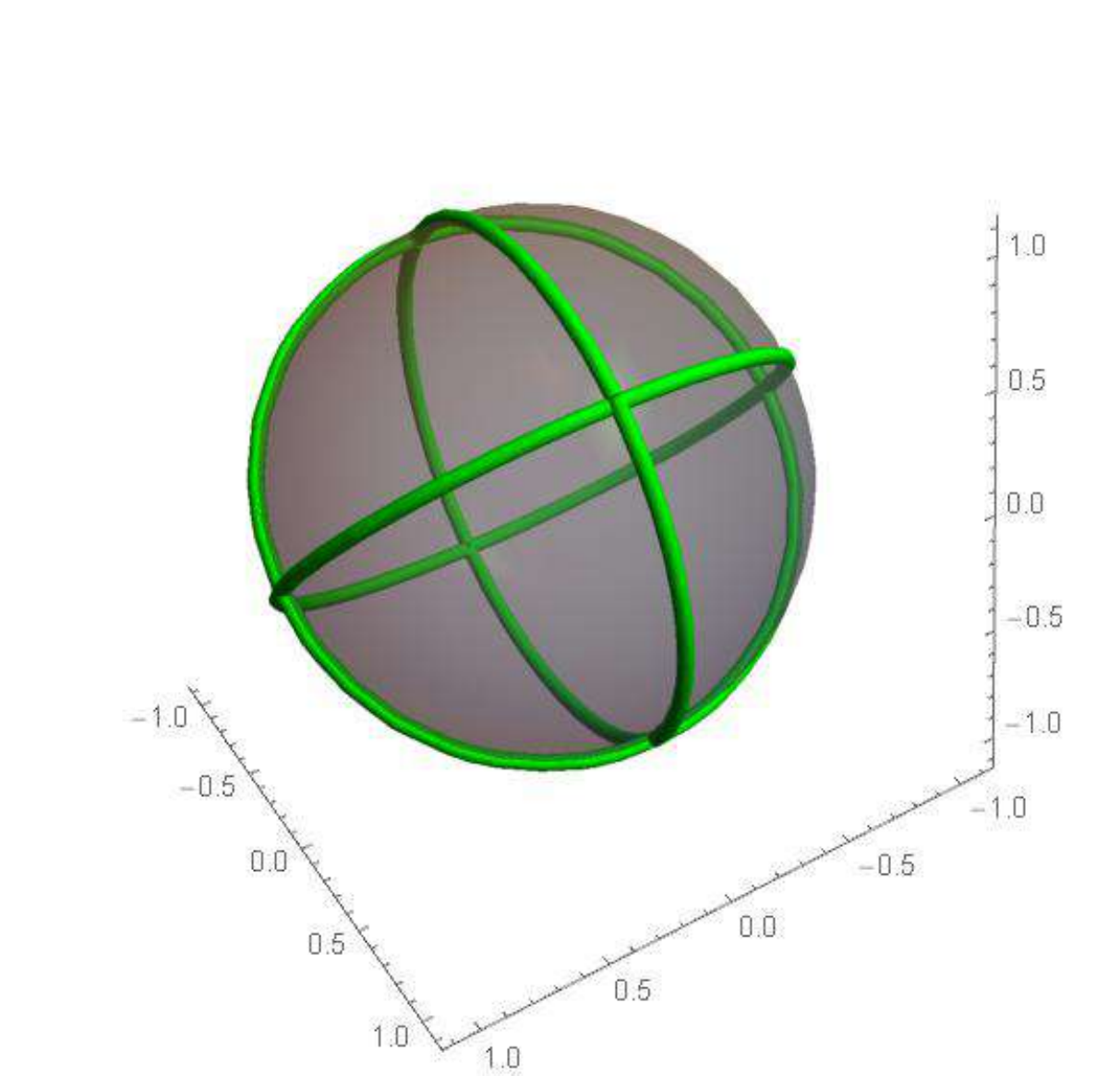}}\\
  \caption{\footnotesize{Phase portrait of the compactified slow-fast systems \eqref{exe-prop-r3-1} and \eqref{exe-prop-r3-infty} in the Poincaré Ball.}}
  \label{fig-exe-prop-r3}
\end{figure}

\section{Geometric singular perturbation theory at infinity}\label{sec-NH-RN}

The main goal of this section is to study conditions in order to assure normal hyperbolicity at infinity of the PL-ball. We start our analysis stating a suitable version of Fenichel Theorem at infinity, which is given in Theorem \ref{teo-fenichel-infty}. Afterwards it is shown that the Newton polytope of a polynomial slow-fast system carries information about the normal hyperbolicity at infinity. Finally, it is given a geometric condition based on the intersection of the critical manifold with infinity that assures normal hyperbolicity.

In Theorem \ref{teo-fenichel-infty}, suppose that $n \geq 3$, consider the polynomial slow-fast system \eqref{eq-def-slowfast-2} and suppose that $C_{0}^{\infty}$ is the critical manifold of the compactified system \eqref{eq-slow-fast-Ul-n} at a generic chart $U_{l}$, for $2\leq l \leq n$. Suppose also that $\delta = \delta_{1} = \delta_{j_{0}}$ for some $j_{0}$ (see item (c) of Proposition \ref{prop-general-dim}).

\begin{theorem}\label{teo-fenichel-infty}
(Fenichel Theorem at infinity) Let $\mathcal{K}\subset\mathcal{NH}(C_{0}^{\infty})$ be a $j$-dimensional compact normally hyperbolic sub-manifold (possibly with boundary) at infinity $\{v_{n} = 0\}$ of the slow system associated to \eqref{eq-slow-fast-Ul-n}. Let $\mathcal{W}^{s}$ be the $(j + j^{s})$-dimensional stable manifold of $\mathcal{K}$. Then, at infinity $\{v_{n} = 0\}$, there is $\tilde{\varepsilon}$ sufficiently small such that for $\varepsilon < \tilde{\varepsilon}$ the following hold:

\begin{description}
    \item[(F1)] There exists a family of smooth manifolds $\mathcal{K}_{\varepsilon}$ such that $\mathcal{K}_{\varepsilon} \rightarrow \mathcal{K}_{0} = \mathcal{K}$ according to Hausdorff distance and $\mathcal{K}_{\varepsilon}$ is a normally hyperbolic locally invariant manifold of \eqref{eq-slow-fast-Ul-n};
    \item[(F2)] There is a family of $(j + j^{s} + k^{s})$-dimensional manifolds $\mathcal{W}^{s}_{\varepsilon}$ such that $\mathcal{W}^{s}_{\varepsilon}$ is local stable manifolds of $\mathcal{K}_{\varepsilon}$.
\end{description}

Analogous conclusions hold for the $(j + j^{u})$-dimensional unstable manifold $\mathcal{W}^{u}$ at infinity.
\end{theorem}

\begin{example}
Consider the slow-fast system
\begin{equation}\label{exe-prop-r3-infty}
 x' = P(x,y,z,\varepsilon) = x(y^{2} - z^{2}); \quad y' = \varepsilon z^{3}; \quad z' = \varepsilon y^{3}.
\end{equation}

After a Poincaré-compactification (PL-compactification with weights $\omega = (1,1,1)$), the systems obtained in $U_{1}$, $U_{2}$ and $U_{3}$ are, respectively,
\begin{equation}
 u' = -u^3 + u v^2 + \varepsilon v^3; \quad v' = \varepsilon u^3 - u^2 v + v^3; \quad w' = w(v^{2} - u^{2});   
\end{equation}

\begin{equation}\label{eq-exe-fenichel-infty-2}
 u' = u(1 - v^{2} - \varepsilon v^{3}); \quad v' = \varepsilon (1-v^{4}); \quad w' = -\varepsilon v^{3}w;   
\end{equation}

\begin{equation}\label{eq-exe-fenichel-infty-3}
 u' = u(v^{2} - 1 - \varepsilon v^{3}); \quad v' = \varepsilon (1-v^{4}); \quad w' = -\varepsilon v^{3}w.    
\end{equation}

Observe that are two non normally hyperbolic points for both systems \eqref{eq-exe-fenichel-infty-2} and \eqref{eq-exe-fenichel-infty-3}. Theorem \ref{teo-fenichel-infty} assures that, at infinity $\{w = 0\}$, the dynamics near compact normally hyperbolic sets persist for $\varepsilon > 0$ sufficiently small. See Figure \ref{fig-exe-prop-r3}.
\end{example}

\subsection{Newton polytope of a polynomial vector field}\label{subsec-newton-polytope}

This subsection is devoted to recall the classical definition of Newton polytope associated to a polynomial vector field (see also \cite{Kappos}). Let $Y = (F_{1},\dots,F_{n})$ be a $n$-dimensional polynomial vector field. For each component $F_{i}$ of $Y$, we introduce the notation

$$
\mathbf{a} = (a_{1},\dots,a_{n}); \quad \mathbf{x} = (x_{1},\dots,x_{n}); \quad \mathbf{x}^{\mathbf{a}} = x_{1}^{a_{1}}\cdot\ldots\cdot x_{n}^{a_{n}};
$$
$$
F_{i}(\mathbf{x}) = \displaystyle\sum_{\mathbf{a}_{i}\in\mathbb{Z}^{n}}c_{\mathbf{a}_{i}}\mathbf{x}^{\mathbf{a}}; \quad c_{\mathbf{a}_{i}}\in\mathbb{R}; \quad 
\mathbf{a}_{i} = (a_{1},\dots,a_{i-1},a_{i}-1,a_{i+1},\dots,a_{n}).
$$

Let $Y = (F_{1},\dots,F_{n})$ be a $n$-dimensional polynomial vector field. The \emph{support of $Y$} is the set $\mathcal{S}_{Y}$ given by
$\mathcal{S}_{Y} = \displaystyle\bigcup_{i = 1}^{n}\mathcal{S}_{Y,i}$, in which $\mathcal{S}_{Y,i} = \{\mathbf{a}_{i}\in\mathbb{Z}^{n}; \ c_{\mathbf{a}_{i}}\neq 0\}$. The \emph{Newton polytope} $\mathcal{P}_{Y}\subset\mathbb{R}^{n}$ of a $n$-dimensional polynomial vector field $Y$ is the convex hull of the support $\mathcal{S}_{Y}$.

\begin{example}\label{exe-newton-planar}
Consider the planar polynomial vector field $Y(x,y) = \big{(}F_{1}(x,y),F_{2}(x,y)\big{)} = (x+y,x^{2})$. It follows that $\mathcal{S}_{Y,1} = \{(0,0),(-1,1)\}$ and $\mathcal{S}_{Y,2} = \{(2,-1)\}$. Therefore $\mathcal{S}_{Y} = \{(0,0),(-1,1),(2,-1)\}$. See Figure \ref{fig-newton-polytope} (a).
\end{example}

\begin{example}\label{exe-newton-three-d}
Consider the $3$-dimensional polynomial vector field $$Y(x,y,z) = \big{(}F_{1}(x,y,z),F_{2}(x,y,z),F_{3}(x,y,z)\big{)} = (-1 + xy,yz^{2},xz).$$
It follows that $\mathcal{S}_{Y,1} = \{(-1,0,0),(0,1,0)\}$, $\mathcal{S}_{Y,2} = \{(0,0,2)\}$ and $\mathcal{S}_{Y,3} = \{(1,0,0)\}$. Therefore $\mathcal{S}_{Y} = \{(-1,0,0),(1,0,0),(0,1,0),(0,0,2)\}$. See Figure \ref{fig-newton-polytope} (b).
\end{example}

\begin{figure}[h!]
  \begin{overpic}[scale=1.3]{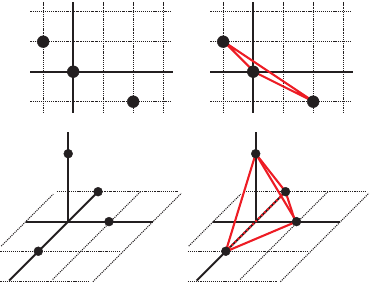}
		\put(-10,55){(a)}
        \put(-10,10){(b)}
		\end{overpic}
  \caption{\footnotesize{Figure (a): Support (left) and Newton polytope (right) of the planar polynomial vector field of the Example \ref{exe-newton-planar}. Figure (b): Support (left) and Newton polytope (right) of the 3-dimensional polynomial vector field of the Example \ref{exe-newton-three-d}.}}
  \label{fig-newton-polytope}
\end{figure}

\subsection{Normal hyperbolicity at infinity}

We have already stated all preliminary definitions and results needed to prove our main results. In what follows, it will be studied necessary conditions in order to assure normal hyperbolicity at infinity. Denote $P(x,\mathbf{y},\varepsilon) = \displaystyle\sum_{d = -1}^{\delta_{1}}P_{ d}(x,\mathbf{y},\varepsilon)$, which $P_{d}$ is a quasi homogeneous polynomial of type $\omega$ of degree $d+\omega_{1}$ and $\delta_{1}+\omega_{1} =\deg_{\omega}{P}$. The degree of quasi homogeneity type $\omega$ of the slow-fast system \eqref{eq-def-slowfast-2} is denoted by $\deg_{\omega}X_{\varepsilon} = \delta \geq \delta_{1}$.

\begin{mtheorem}\label{teo-a} Let $X_{\varepsilon}$ be a $n$-dimensional polynomial vector field associated to the slow-fast system \eqref{eq-def-slowfast-2}, whose critical manifold is given by $C_{0} = \{P(x,\mathbf{y},0) = 0\}$. Then, near the boundary of the PL ball the following hold:

\begin{description}

\item[(a)] Suppose $\delta = \delta_{1}$. If the component $P$ of type $\omega$ and degree $\delta_{1} + \omega_{1}$ has monomials of the form
$$c_{1}x^{r_{1}+1} + \displaystyle\sum_{i = 2}^{n}\big{(}c_{i}xy_{i}^{r_{i}} + d_{i}y_{i}^{s_{i}}\big{)},$$
satisfying $c_{i}^{2} + d_{i}^{2} \neq 0$ for all $i = 1,\dots,n$; and $(r_{1},r_{2},\dots,r_{n})$ and $(r_{1},s_{2},\dots,s_{n})$ satisfy the equation of the hyperplane $\{\omega_{1}a_{1} + \ldots + \omega_{n}a_{n} = \delta\}$, then the origin of each chart $U_{l}$ is a normally hyperbolic point of the critical manifold. In particular, the origin of $U_{1}$ is a hyperbolic node of the compactified vector field.
\item[(b)] If $\delta > \delta_{1}$, the infinity $\{v_{n} = 0\}$ is a non normally hyperbolic component of the critical manifold $C_{0}$.
\item[(c)] If $p\in U_{l}$ is a point at infinity (with $2 \leq l \leq n$), a necessary condition to assure normal hyperbolicity is that $C_{0}$ intersects the infinity $\{v_{n} = 0\}$ transversely at $p$.
\end{description}
\end{mtheorem}

\begin{remark}
The hypotheses of Condition (a) of Theorem \ref{teo-a} implies that the Newton polytope $\mathcal{P}_{X_{\varepsilon}}$ has a compact face containing the intersection of the hyperplane $\{\omega_{1}a_{1} + \ldots + \omega_{n}a_{n} = \delta\}$ with $(\mathbb{R}_{\geq 0})^{n}$. Such feature of the Newton polytope turns out to be a necessary condition in order to assure normal hyperbolicity at the origin of each chart at infinity. From a practical way of view, one can use the Newton polytope in order to detect non normally hyperbolic points of the origin of each chart. Finally, observe that in condition (c) we do not require that the point $p$ is the origin.
\end{remark}
\begin{proof}

From Proposition \ref{prop-general-dim}, we know that the compactification $X^{\infty}_{ \varepsilon}$ defines a slow fast system in the charts $U_{l}$ for $l = 2,\dots,n$, but it is not in the chart $U_{1}$. From Equation \eqref{eq-slow-fast-Ul-n}, we obtain the expression of the critical manifold $C_{0} = \{P(x,\mathbf{y},0) = 0\}$ in the chart $U_{l}$:
$$
\displaystyle\sum_{d = 0}^{\delta}v^{\delta-d}_{n}P_{ d}\big{(}u,\dots,v_{l-1}, 1,v_{l},\dots,v_{n-1},0\big{)} = 0.
$$

Therefore, normal hyperbolicity near infinity $\{v_{n} = 0\}$ means
\begin{equation}\label{eq-sis-NH-inf}
v^{\delta-\delta_{1}}_{n}P_{ \delta_{1}} =  0, \ \quad \ 
v^{\delta-\delta_{1}}_{n}\displaystyle\frac{\partial P_{ \delta_{1}}}{\partial u} \neq 0;
\end{equation}
in which such functions are applied in $(u,\dots,v_{l-1}, 1,v_{l},\dots,v_{n-1},0)$. So we divide our analysis in two cases.

\textbf{(a)} Suppose that $\delta = \delta_{1}$. In this case, a necessary condition to assure that the origin of the chart $U_{l}$ is normally hyperbolic is to require that the original polynomial $P$ has monomials of the form $c_{l}xy^{r_{l}}_{l} + d_{l}y^{s_{l}}_{l}$, in which $c_{l},d_{l}\in\mathbb{R}$, $r_{l},s_{l}\in\mathbb{N}$ and $c_{l}^{2} + d_{l}^{2} \neq 0$, $r_{l} = \frac{\delta}{\omega_{l}}$ and $s_{l} = \frac{\delta +\omega_{1}}{\omega_{l}}$.

Indeed, if $d_{l} \neq 0$, then the critical set $C_{0}^{\infty}$ does not intersect the origin of the chart $U_{l}$. On the other hand, if $d_{l} = 0$ then $c_{l} \neq 0$ and the critical set $C_{0}^{\infty}$ is normally hyperbolic at the origin of $U_{l}$. Recall that $C_{0}^{\infty}$ is the critical manifold of the compactified system \eqref{eq-slow-fast-Ul-n} at a generic chart $U_{l}$, for $2\leq l \leq n$.

Concerning the support $\mathcal{S}_{X_{\varepsilon}}$, if $c_{l} \neq 0$ then $\mathcal{S}_{X_{\varepsilon}}$ contains the point $(0,\dots,0,r_{l},0,\dots,0)$, in which $r_{l}$ is positioned in the $l$-th coordinate. Finally, if $d_{l} \neq 0$ then $\mathcal{S}_{X_{\varepsilon}}$ contains the point $(-1,0,\dots,0,s_{l},0,\dots,0)$, in which $s_{l}$ is positioned in the $l$-th coordinate. 

The compactification of the critical manifold may be normally hyperbolic in one chart and not be in another chart. Therefore, in order to assure that, for all $l = 2,\dots,n$ the origin of $U_{l}$ is normally hyperbolic, we must require that the quasi homogeneous component of $P$ of type $\omega$ and degree $\delta_{1} + \omega_{1}$ has monomial of the form
$$x\Big{(}c_{1}x^{r_{1}} + c_{2}y_{2}^{r_{2}} + \dots + c_{n}y_{n}^{r_{n}}\Big{)} + \Big{(}d_{2}y_{2}^{s_{2}} + \dots + d_{n}y_{n}^{s_{n}} \Big{)},$$
in which $c_{i}^{2} + d_{i}^{2} \neq 0$ for all $i = 1,\dots,n$.

Observe that the points of $\mathcal{S}_{X_{\varepsilon}}$ related to these monomials are contained in the hyperplane $\{\omega_{1}a_{1} + \ldots + \omega_{n}a_{n} = \delta\}$. Moreover, since all the natural numbers $r_{i}$ and $s_{i}$ concerns higher order terms of the vector field $X_{\varepsilon}$, then all the other points of the support $\mathcal{S}_{X_{\varepsilon}}$ are either contained in such hyperplane, or they are contained in the half-space $\{\omega_{1}a_{1} + \ldots + \omega_{n}a_{n} < \delta\}$. This implies that the Newton polytope $\mathcal{P}_{X_{\varepsilon}}$ has a compact face that contains in the hyperplane $\{\omega_{1}a_{1} + \ldots + \omega_{n}a_{n} = \delta\}\cap(\mathbb{R}_{\geq 0})^{n}$. See Figure \ref{fig-newton-NH}. If this is the case, in the chart $U_{1}$ the origin is a hyperbolic node of the compactified vector field.

\begin{figure}[h!]
  \begin{overpic}[scale=1.3]{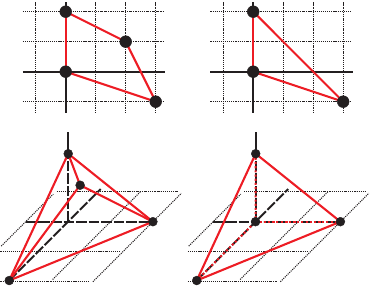}
		\put(-10,55){(a)}
        \put(-10,10){(b)}
		\end{overpic}
  \caption{\footnotesize{Figure (a): On the left, an example of Newton polytope of that gives rise to non normally hyperbolic points at infinity. On the other hand, the slow-fast system associated to the Newton polytope on the right will present normally hyperbolic points at the origin of $U_{1}$ and $U_{2}$ (under a suitable choice of $\omega$). Figure (b): description analogous to the Figure (a), but for the three dimensional case.}}
  \label{fig-newton-NH}
\end{figure}

\textbf{(b)} Suppose that $\delta > \delta_{1}$. From \eqref{eq-sis-NH-inf}, it is clear that, for each chart $U_{l}$, the infinity $\{v_{n} = 0\}$ is a non normally hyperbolic component of the critical manifold.

\textbf{(c)} Denote by $C_{0}^{\infty}$ the critical manifold $C_{0}$ in the chart $U_{l}$, for each $l = 2,\dots,n$. The infinity is represented by the hyperplane $\{v_{n} = 0\}$, and the vector $(0,\dots,0,1)$ is normal to such hyperplane, for every point $p\in\{v_{n} = 0\}$. On the other hand, the vector $\nabla P_{ \delta}(p)$ is normal to the critical manifold at $p$. If the critical manifold is normally hyperbolic at $p$, from equation \eqref{eq-sis-NH-inf} we know that the first coordinate of $\nabla P_{ \delta}(p)$ is non zero. Therefore, $(0,\dots,0,1)$ and $\nabla P_{ \delta}(p)$ are linearly independent, which implies that $T_{p}C_{0}^{\infty}\pitchfork T_{p}\{v_{n} = 0\}$. We conclude that a necessary condition to assure normal hyperbolicity at infinity is that the critical manifold intersects the infinity transversely. See Figure \ref{fig-NH-transversal}.
\end{proof}

\begin{figure}[h!]
  \begin{overpic}[scale=0.9]{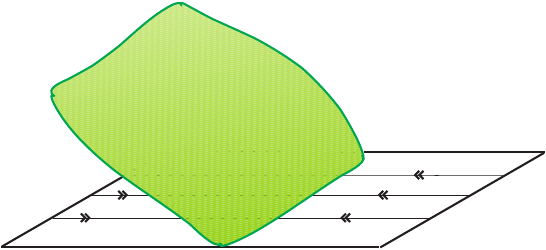}
        \put(-15,5){$\{v_{n} = 0\}$}
        \put(50,40){$C_{0}^{\infty}$}
		\end{overpic}
  \caption{\footnotesize{Critical manifold $C_{0}^{\infty}$ (highlighted in green) intersects the infinity transversally.}}
  \label{fig-NH-transversal}
\end{figure}

\begin{remark}
From Propositon \ref{prop-general-dim} item (d), in each chart $U_{l}$ we know that, if $\delta > \delta_{1}$, by setting $\varepsilon = 0$, the infinity $\{v_{n} = 0\}$ is filled with equilibrium points. Due to item (b) of Theorem \ref{teo-a}, now we know that, in fact, the infinity is a component of the critical manifold $C_{0}$.
\end{remark}

In what follows we present an example showing that transversality is not a sufficient condition to assure normal hyperbolicity at infinity.

\begin{example}
Consider the polynomial slow-fast system
\begin{equation}\label{eq-exe-non-transversal}
x' = y + z, \quad y' = \varepsilon(x+z), \quad z' = \varepsilon(x+y).
\end{equation}

The critical manifold associated to \eqref{eq-exe-non-transversal} is the plane $C_{0} = \{y+z = 0\}$, which intersects the infinty transversaly. However, after PL - compactification with weight $\omega = (1,1,1)$, one obtains the following slow-fast system in both charts $U_{2}$ and $U_{3}$:
$$
u' = 1  + v - \varepsilon u(u + v), \quad v' = \varepsilon(1 + u - uv - v^{2}), \quad
    w' = - \varepsilon w(u + v).
$$

In both charts in $U_{2}$ and $U_{3}$, the infinity and the critical manifold are given by $\{w = 0\}$ and $C_{0}^{\infty} = \{v + 1 = 0\}$, respectively. This implies that $\mathcal{NH}(C_{0}^{\infty}) = \emptyset$. Geometrically, in $U_{2}$ and $U_{3}$ the set $C_{0}^{\infty}$ is a horizontal line. See Figure \ref{fig-exe-non-transversal}.
\end{example}

\begin{figure}[h!]
	\center{\includegraphics[width=0.55\textwidth]{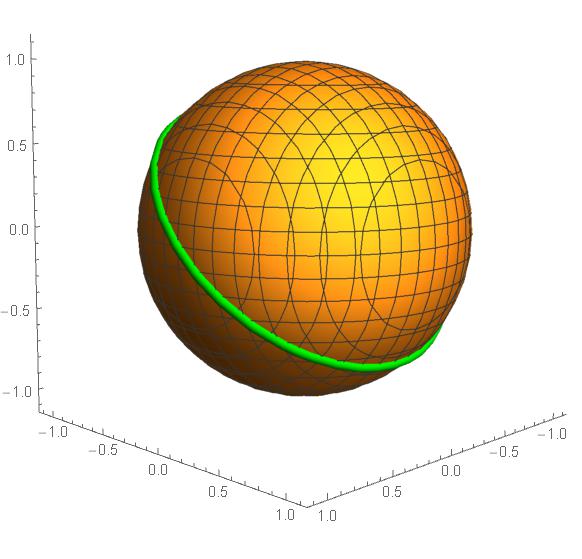}}
	\caption{\footnotesize{Compactification of the critical manifold of the slow fast system \eqref{eq-exe-non-transversal}. The critical manifold $C_{0}$ is highlighted in green.}}
	\label{fig-exe-non-transversal}
\end{figure}

It is very difficult to study conditions to assure normal hyperbolicity for the \emph{whole} infinity in arbitrary dimension. However, it can be given an answer for the $2$-dimensional case (see Theorem \ref{teo-fenichel-global-r2}). For instance, it will be clear that, in dimension 2, the transversality condition presented in Theorem \ref{teo-a} is sufficient and necessary to assure normal hyperbolicity at infinity.

\section{Planar polynomial slow-fast systems}\label{sec-planar-gspt}

\noindent

Consider the 2-dimensional polynomial slow-fast system
\begin{equation}\label{eq-polynomial-slow-fast-r2}
x' = P(x,y,\varepsilon); \ \quad \ y' = \varepsilon Q(x,y,\varepsilon).
\end{equation}

As usual, $P_{ d}$ and $Q_{ d}$ are the quasi homogeneous component of type $\omega$ and degree $d+\omega_{1}$ and $d+\omega_{2}$, respectively. The highest quasihomogeneous degree component of $P$ and $Q$ is, respectively, $P_{ \delta_{1}}$ and $Q_{ \delta_{2}}$. Due to statements (a) and (b) of Proposition \ref{prop-general-dim}, for our purposes in this section we will further suppose that $\delta_{1} = \delta_{2} = \delta$. 

The polynomial functions $P$ and $Q$ will be written as
$$P(x,y,\varepsilon) = \displaystyle\sum_{i = -1}^{\delta}P_{ i}(x,y,\varepsilon), \quad Q(x,y,\varepsilon) = \displaystyle\sum_{j = -1}^{\delta}Q_{ j}(x,y,\varepsilon);$$
$$P_{ i}(x,y,\varepsilon) = \displaystyle\sum_{r = 0}^{i}c_{\varepsilon,r,i}x^{r}y^{\frac{i+\omega_{1}(1-r)}{\omega_{2}}}, \quad Q_{ j}(x,y,\varepsilon) = \displaystyle\sum_{s = 0}^{j}d_{\varepsilon,s,j}x^{s}y^{\frac{j+\omega_{2}-\omega_{1}s}{\omega_{2}}};$$
in which $\delta = \deg_{\omega}X_{\varepsilon}$ and the notation $c_{\varepsilon,r,i},\ d_{\varepsilon,s,j}$ indicates that the coefficients of $P$ and $Q$ depend analytically on $\varepsilon$. Observe that, for each $i$ and $j$, the powers of the monomials of $P_{ i}(x,y,\varepsilon)$ and $Q_{ j}(x,y,\varepsilon)$ satisfies, respectively, $a\omega_{1}+b\omega_{2} = i + \omega_{1}$ and $a\omega_{1}+b\omega_{2} = j + \omega_{2}$.

The compactification $X^{\infty}_{ \varepsilon}$ in the fast and slow direction is given by, respectively,
\begin{equation}\label{eq-sys-compact-U1}
\left\{
  \begin{array}{rcl}
   u' & = & \displaystyle\sum_{i = -1}^{\delta}v^{\delta-i}\Big{(}-uP_{ i}(1,u,\varepsilon) + \varepsilon Q_{ i}(1,u,\varepsilon)\Big{)},  \\
   v' & = & -\displaystyle\sum_{i = -1}^{\delta}v^{\delta-i+1}P_{ i}(1,u,\varepsilon); 
  \end{array}
\right.
\end{equation}
\begin{equation}\label{eq-sys-compact-U2}
\left\{
  \begin{array}{rcl}
   u' & = & -\displaystyle\sum_{i = -1}^{\delta}v^{\delta-i}\Big{(}-u\varepsilon Q_{ i}(u,1,\varepsilon) +  P_{ i}(u,1,\varepsilon)\Big{)}  \\
    v' & = & -\varepsilon\displaystyle\sum_{i = -1}^{\delta}v^{\delta-i+1}Q_{ i}(u,1,\varepsilon).
  \end{array}
\right.
\end{equation}

Before we start our analysis in the charts $U_{1}$ and $U_{2}$, let us introduce an useful Lemma, that can be found for instance in \cite[pp.~72]{GarciaLequain}. Recall that $x_{0}$ is a \emph{simple root} of a polynomial $P\in\mathbb{R}[x]$ if $P(x) = (x - x_{0})\cdot \widetilde{P}(x)$ and $x_{0}$ is not a root of $\widetilde{P}(x)$.

\begin{lemma}\label{lemma-simple-root}
Let $P\in\mathbb{R}[x]$. Then $x_{0}$ is a simple root of $P$ if, and only if, $P(x_{0}) = 0$ and $P'(x_{0})\neq 0$.
\end{lemma}


Let us start our
study by considering the compactification in the slow direction \eqref{eq-sys-compact-U2}.

\begin{proposition}\label{prop-nh-intersection}
Let \eqref{eq-polynomial-slow-fast-r2} be a planar polynomial slow-fast system and consider its PL-compactification in the slow direction \eqref{eq-sys-compact-U2}. Then $(\tilde{u},0)$ is an equilibrium point at infinity $\{v = 0\}$ for $\varepsilon = 0$ if, and only if, the critical manifold $C_{0}$ intersects the infinity at $(\tilde{u},0)$. 
\end{proposition}
\begin{proof}
Recall that $\delta = \deg_{\omega}X_{\varepsilon}$. Therefore, in the chart $U_{2}$ the critical manifold $C_{0}$ is given by $\{P_{ \delta}(u,1,0) = 0\}$, which is a curve of equilibria of system \eqref{eq-sys-compact-U2} when $\varepsilon = 0$. Such a curve intersects the infinity at points of the form $(\tilde{u},0)$, with $\tilde{u}$ being a root of the polynomial $P_{ \delta}(u,1,0)$, which concludes the proof.
\end{proof}

\begin{proposition}\label{prop-planar-equivalence-NH}
Let $p = (\tilde{u},0)\in\{v = 0\}$ an equilibrium point of \eqref{eq-sys-compact-U2} positioned at infinity. Then the following statements are equivalent:
\begin{description}
    \item[(a)] $p$ is normally hyperbolic for \eqref{eq-sys-compact-U2} when $\varepsilon = 0$.
    \item[(b)] $\tilde{u}$ is a hyperbolic equilibrium point of the ODE $\dot{u} = P_{ \delta}(u,1,0)$.
    \item[(c)] $\tilde{u}$ is a simple root of the polynomial $P_{ \delta}(u,1,0)$.
    \item[(d)] The critical manifold intersects the infinity transversely.
\end{description}
\end{proposition}
\begin{proof}
Observe that (a) $\Leftrightarrow$ (b) because $\frac{\partial}{\partial u}P_{ \delta}(\tilde{u},1,0) \neq 0$ if, and only if, $\tilde{u}$ is a hyperbolic equilibrium point of the ODE $\dot{u} = P_{ \delta}(u,1,0)$. In addition, from Lemma \ref{lemma-simple-root}, it follows that (a) $\Leftrightarrow$ (c) because $\frac{\partial}{\partial u}P_{ \delta}(\tilde{u},1,0) \neq 0$ and  $P_{ \delta}(\tilde{u},1,0) = 0$ if, and only if, $\tilde{u}$ is a simple root of $P_{ \delta}(u,1,0)$. The equivalence (a) $\Leftrightarrow$ (d) follows because $\frac{\partial}{\partial u}P_{ \delta}(\tilde{u},1,0) \neq 0$ means that $\nabla P_{ \delta}(\tilde{u},1,0)$ and $\vec{n} = (0,1)$ are linearly independent, in which $\vec{n}$ is normal to the line that represents the infinity.
\end{proof}

Due to Propositions \ref{prop-nh-intersection} and \ref{prop-planar-equivalence-NH}, the transversality condition presented in Theorem \ref{teo-a} is a sufficient and necessary condition to assure normal hyperbolicity at infinity in the 2-dimensional case.

\begin{corollary}
A necessary condition for the existence of simple roots of $P_{ \delta}(u,1,0)$ is $c_{0,0,\delta} \neq 0$ or $c_{0,1,\delta} \neq 0$.
\end{corollary}

\begin{proof}
From the expression of $P_{i}(x,y,\varepsilon)$, if $c_{0,0,\delta} \neq 0$, then the component $P_{ \delta}$ of $P$ has a monomial of the form $c_{0,0,\delta}y^{\frac{\delta+\omega_{1}}{\omega_{2}}}$. On the other hand, if, $c_{0,1,\delta} \neq 0$, then the component $P_{ \delta}$ of $P$ has a monomial of the form $c_{0,1,\delta}xy^{\frac{\delta}{\omega_{2}}}$. In both cases, by setting $(u,1,0)$, the origin of the chart $U_{2}$ will be either a regular point or a hyperbolic equilibrium point of the the compactification in the slow direction \eqref{eq-sys-compact-U2}.
\end{proof}

Since we have studied the dynamics of the compactification in the slow direction, it is sufficient to study the dynamics near the origin of the compactification at the fast direction \eqref{eq-sys-compact-U1}. Recall that \eqref{eq-sys-compact-U1} is not a slow-fast system.

\begin{proposition}\label{prop-r2-u1}
Let \eqref{eq-polynomial-slow-fast-r2} be a planar polynomial slow-fast system and consider its compactification in the fast direction \eqref{eq-sys-compact-U1}. Then, for $\varepsilon = 0$, the following statements are true
\begin{description}
    \item[(a)] The origin of the chart $U_{1}$ is an equilibrium point of \eqref{eq-sys-compact-U1}.
    \item[(b)] The critical manifold $C_{0}$ intersects the origin of $U_{1}$ if, and only if, $c_{0,\frac{\delta}{\omega_{1}},\delta} = 0$. In this case, the origin is a non-hyperbolic equilibrium point of \eqref{eq-sys-compact-U1}.
    \item[(c)] The origin of $U_{1}$ is an hyperbolic node of \eqref{eq-sys-compact-U1} if, and only if, $c_{0,\frac{\delta}{\omega_{1}},\delta} \neq 0$.
\end{description}
\end{proposition}
\begin{proof}
Recall the expression of $P_{i}(x,y,\varepsilon)$. It is straightforward from equation \eqref{eq-sys-compact-U1} that the origin of $U_{1}$ is an equilibrium point for $\varepsilon = 0$. Therefore, items (b) and (c) aims to understand features of such equilibrium. Moreover, in this chart we must study the polynomial $P$ applied in points of the form $(1,u,0)$.

Observe that in the chart $U_{1}$ the critical manifold $C_{0}$ is the zero set of the polynomial $P_{ \delta}(1,u,0)$, which represents a curve of singularities for $\varepsilon = 0$. Therefore, the critical manifold intersects the origin of $U_{1}$ if, and only if, $c_{0,\frac{\delta}{\omega_{1}},\delta} = 0$. Moreover, the origin will be non hyperbolic for \eqref{eq-sys-compact-U1} if such a point is contained in $C_{0}$. This proves item (b) Finally, assuming that $\varepsilon = 0$, it can be easily checked that $c_{0,\frac{\delta}{\omega_{1}},\delta} \neq 0$ if, and only if, the origin of $U_{1}$ is an hyperbolic node of \eqref{eq-sys-compact-U1}, which proves item (c).
\end{proof}

Now, we are able to state, for the planar case, a \emph{global} version of the Fenichel Theorem, which assures the persistence of invariant manifolds in the \emph{whole} Poincaré--Lyapunov disk. The proof of Theorem \ref{teo-fenichel-global-r2} is given by combining Fenichel Theorem (for the finite part) and Propositions \ref{prop-nh-intersection}, \ref{prop-planar-equivalence-NH} and \ref{prop-r2-u1} (for the infinite part). In its statement, the compactified critical manifold is denoted by $\mathbf{C}_{0}$, which is the union of the finite and infinite parts of $C_{0}$.

\begin{mtheorem}\label{teo-fenichel-global-r2}
Consider the planar polynomial slow-fast system \eqref{eq-polynomial-slow-fast-r2}. Suppose that $\mathcal{NH}(C_{0}) = C_{0}$, $C_{0}$ intersects the infinity of $\mathbb{D}_{\omega}$ transversely, and it does not intersect the origin of $U_{1}$, $V_{1}$. Then there exist $0 < \tilde{\varepsilon} \ll 1$ such that for $\varepsilon < \tilde{\varepsilon}$ the following hold in the whole $\mathbb{D}_{\omega}$:
\begin{description}
    \item[(G1)] There exist a family of smooth manifolds $\mathbf{C}_{\varepsilon}$ such that $\mathbf{C}_{\varepsilon}\rightarrow \mathbf{C}_{0}$ according to Hausdorff distance and $\mathbf{C}_{\varepsilon}$ is locally invariant of \eqref{eq-polynomial-slow-fast-r2}.
    \item[(G2)] If $p_{0}\in\mathbf{C}_{0}$ and $\mathbf{W}^{s}$ is its stable manifold, then there is a family $\mathbf{W}^{s}_{\varepsilon}$ of stable manifolds of $p_{\varepsilon}\in\mathbf{C}_{\varepsilon}$, in which $p_{\varepsilon}\rightarrow p_{0}$. The same conclusion holds if one consider the unstable manifold $\mathbf{W}^{u}$ of $p_{0}\in\mathbf{C}_{0}$.
\end{description}

\end{mtheorem}

\begin{example}
Consider the slow fast system
\begin{equation}\label{eq-exe-slow-fast-dim-2}
x' = x^{2} + xy - 1, \ \quad \ y' = \varepsilon Q(x,y,\varepsilon);
\end{equation}
in which $Q(x,y,\varepsilon)$ is a polynomial function of degree equal or lesser than $2$. A suitable PL-compactification for \eqref{eq-exe-slow-fast-dim-2} is made considering the weight vector $\omega = (1,1)$, therefore we apply the classical Poincaré compactification. Observe that $\deg_{\omega}X_{\varepsilon} = \delta = 1$ and $\deg_{\omega} Q = \delta_{2} \leq 1$. In $U_{1}$ and $U_{2}$, the dynamics at infinity are respectively given by
\begin{equation}
u' = u (-1 - u + v^{2}) + \varepsilon v^{2-\delta_{2}}Q(1, u); \ \quad \ v' = v (-1 - u + v^{2});     
\end{equation}
\begin{equation}
u' = u + u^{2} - v^{2} - \varepsilon u v^{2-\delta_{2}} Q(u,1); \ \quad \ v' = -\varepsilon v^{3-\delta_{2}} Q(u,1).    
\end{equation}

Observe that all points in the critical manifold $C_{0} = \{x^{2} + xy = 1\}$ are normally hyperbolic. The origin of the chart $U_{1}$ is a hyperbolic node and $C_{0}$ is normally hyperbolic at infinity (see Proposition \ref{prop-planar-equivalence-NH}). Therefore, as a consequence of Theorem \ref{teo-fenichel-global-r2}, the global dynamics of system \eqref{eq-exe-slow-fast-dim-2} persist for $\varepsilon$ sufficiently small. See Figure \ref{fig-exe-dim-2}.

\begin{figure}[h!]
	\center{\includegraphics[width=0.35\textwidth]{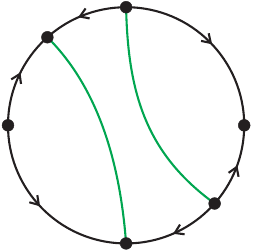}}
	\caption{\footnotesize{Poincaré compactification of system \eqref{eq-exe-slow-fast-dim-2}. The critical manifold $C_{0}$ is highlighted in green and the dots denote equilibria at infinity.}}
	\label{fig-exe-dim-2}
\end{figure}

\end{example}

\section{Non normally hyperbolic points at infinity}\label{sec-non-NH}

In this section is discussed some examples of 3-dimensional polynomial slow fast systems that present non normally hyperbolic singularities at infinity, namely: fold, transcritical and pitchfork singularities (see \cite{KrupaSzmolyan, KrupaSzmolyan2}). Observe that, since the phase-space is 3-dimensional, the slow-fast system at infinity is 2-dimensional.

Firstly, we must recall the normal forms of such singularities in the literature. According to \cite{KrupaSzmolyan, KrupaSzmolyan2}, the non degeneracy conditions that a planar slow-fast system \eqref{eq-def-slowfast-2} must satisfy in order to present (respectively) a fold, transcritical and pitchfork singularity is given (respectively) by the equations \eqref{eq-sing-fold-non-degeneracy-conditions}, \eqref{eq-sing-transcritical-non-degeneracy-conditions} and \eqref{eq-sing-pitchfork-non-degeneracy-conditions}:

\begin{equation}\label{eq-sing-fold-non-degeneracy-conditions}
\begin{split}
    f_{x}(0,0,0) = 0; \ \ f_{xx}(0,0,0) \neq 0; \\  f_{y}(0,0,0) \neq 0 \ \ \text{and} \ \ g(0,0,0)\neq 0.
\end{split}
\end{equation}

\begin{equation}\label{eq-sing-transcritical-non-degeneracy-conditions}
\begin{split}
f(0,0,0) = f_{x}(0,0,0) = f_{y}(0,0,0) = 0; \\
\det\operatorname{Hes}(f) < 0; \ \ f_{xx}(0,0,0)\neq 0\neq g(0,0,0);
\end{split}
\end{equation}
where $\operatorname{Hes}(f)$ denotes the Hessian matrix of $f$.

\begin{equation}\label{eq-sing-pitchfork-non-degeneracy-conditions}
\begin{split}
f(0,0,0) = f_{x}(0,0,0) = f_{xx}(0,0,0) = f_{y}(0,0,0) = 0; \\
f_{xxx}(0,0,0) \neq 0, \ \ f_{xy}(0,0,0) \neq 0, \ \ g(0,0,0)\neq 0.
\end{split}
\end{equation}

Theorem \ref{teo-fenichel-normal-form-fold} gathers the results on normal forms of planar slow-fast systems, based on the non degeneracy conditions above. The notation $\mathcal{O}$ denotes the higher order terms, whereas $\lambda$ denotes a constant that depends on the conditions of non-degeneracy of each singularity (see \cite{KrupaSzmolyan,KrupaSzmolyan2} for details).

\begin{theorem}\label{teo-fenichel-normal-form-fold}
There exists a smooth change of coordinates such that for $(x,y)$ sufficiently small a planar slow-fast system is written as
\begin{description}
    \item[(a)] If system \eqref{eq-def-slowfast-2} satisfies the non-degeneracy conditions \eqref{eq-sing-fold-non-degeneracy-conditions} of a planar generic fold:
    \begin{equation}\label{eq-teo-fenichel-normal-form-fold}
   x' = y + x^{2} + \mathcal{O}(x^{3}, xy, y^{2}, \varepsilon); \ \ \ 
   y' = \varepsilon\Big{(}\pm 1 + \mathcal{O}(x,y,\varepsilon)\Big{)};
\end{equation}
  \item[(b)] If system \eqref{eq-def-slowfast-2} satisfies the non-degeneracy conditions \eqref{eq-sing-transcritical-non-degeneracy-conditions} of a generic transcritical singularity:   \begin{equation}\label{eq-normal-form-transcritical}
    x' = x^{2} - y^{2} + \lambda\varepsilon + \mathcal{O}(x^{3}, x^{2}y, xy^{2},y^{3},\varepsilon x,\varepsilon y,\varepsilon^2); \ \ \ 
    y' = \varepsilon\Big{(}1 + \mathcal{O}(x,y,\varepsilon)\Big{)};
\end{equation}

    \item[(c)] If  system \eqref{eq-def-slowfast-2} satisfies the non-degeneracy conditions \eqref{eq-sing-pitchfork-non-degeneracy-conditions} of a pitchfork singularity: \begin{equation}\label{eq-normal-form-pitchfork}
    x' = x(y - x^{2}) + \lambda\varepsilon + \mathcal{O}(x^{2}y, xy^{2},y^{3},\varepsilon x,\varepsilon y,\varepsilon^2); \ \ \ 
    y' = \varepsilon\Big{(}\pm 1 + \mathcal{O}(x,y,\varepsilon)\Big{)}.
\end{equation}

\end{description}
\end{theorem}

The main goal of this section is to study conditions that a 3-dimensional polynomial slow-fast system of the form
\begin{equation}\label{eq-3d-slow-fast}
x' = P(x,y,z,\varepsilon), \quad y' = \varepsilon Q(x,y,z,\varepsilon), \quad z' = \varepsilon R(x,y,z,\varepsilon)
\end{equation}
must satisfy in order to present a fold, transcritical or pitchfork singularity at infinity in the Poincaré-Lyapunov compactification with weight $\omega = (\omega_{1},\omega_{2},\omega_{3})$. Without loss of generality, in what follows is studied conditions in order to assure that the origin of the chart $U_{2}$ is one of the non normally hyperbolic points given by Theorem \ref{teo-fenichel-normal-form-fold}. Moreover, if $X_{\varepsilon}$ is the vector field associated to \eqref{eq-3d-slow-fast}, then $\deg_{\omega}X_{\varepsilon} = \delta$.

\begin{mtheorem}\label{teo-normal-forms}
Consider the 3-dimensional slow fast system \eqref{eq-3d-slow-fast} and its Poincaré-Lyapunov compactification $X^{\infty}_{ \varepsilon}$ with weight $\omega = (\omega_{1},\omega_{2},\omega_{3})$. If $\deg_{\omega}X_{\varepsilon} = \delta$, then the following hold for every positive integer $k_{1},k_{2}$:

\begin{description}
    \item[(a)] If $P_{ \delta}(x,y,z,\varepsilon) = x^{2}y^{k_{1}} - y^{k_{2}}z$, then the critical manifold of the compactified vector field $X^{\infty}_{ \varepsilon}$ has a fold singularity at the origin of the chart $U_{2}$ if, and only if, $k_{1}\omega_{2} = \delta - \omega_{1}$ and $k_{2}\omega_{2} = \delta + \omega_{1} -
    \omega_{3}$.

    \item[(b)] If $P_{ \delta}(x,y,z,\varepsilon) = x^{2}y^{k_{1}} - y^{k_{2}}z^{2}$, then the critical manifold of the compactified vector field $X^{\infty}_{ \varepsilon}$ has a transcritical singularity at the origin of the chart $U_{2}$ if, and only if, $k_{1}\omega_{2} = \delta - \omega_{1}$ and $k_{2}\omega_{2} = \delta + \omega_{1} -
    2\omega_{3}$.
    \item[(c)] If $P_{ \delta}(x,y,z,\varepsilon) = xy^{k_{1}}z - x^{3}y^{k_{2}}$, then the critical manifold of the compactified vector field $X^{\infty}_{ \varepsilon}$ has a pitchfork singularity at the origin of the chart $U_{2}$ if, and only if, $k_{1}\omega_{2} = \delta - \omega_{3}$ and $k_{2}\omega_{2} = \delta - 2\omega_{1}$.
\end{description}
\end{mtheorem}
\begin{proof}
The proof is given by straightforward computations. We present the calculations of the proof of item (a), since the calculations of the other items are completely analogous.

Consider the 3-dimensional polynomial slow-fast system \eqref{eq-3d-slow-fast}. We recall from item (c) of Proposition \ref{prop-general-dim} that, the vector field in the chart $U_{2}$ is a slow-fast system if, and only if, $\deg_{\omega}X_{\varepsilon} = \deg_{\omega}P = \deg_{\omega}Q$ or $\deg_{\omega}X_{\varepsilon} = \deg_{\omega}P = \deg_{\omega}R$.

Given the weight vector $\omega = (\omega_{1},\omega_{2},\omega_{3})$, the expression of the compactified slow-fast system in the chart $U_{2}$ is
\begin{equation}\label{eq-teo-non-nh-fold}
\left\{
  \begin{array}{rcl}
   u' & = & w^{\omega_{1}}P - \varepsilon u \frac{\omega_{1}}{\omega_{2}}w^{\omega_{2}}Q,  \\
   v' & = & \varepsilon\big{(}w^{\omega_{3}}R -  v \frac{\omega_{3}}{\omega_{2}}w^{\omega_{2}}Q\big{)},  \\
   w' & = & -\varepsilon\frac{w^{\omega_{2}+1}}{\omega_{2}}Q,
  \end{array}
\right.
\end{equation}
in which $P,Q$ and $R$ are applied in $(\frac{u}{w^{\omega_{1}}},\frac{1}{w^{\omega_{2}}},\frac{v}{w^{\omega_{3}}},\varepsilon)$. 

Given that the highest quasihomogeneous degree component of $P$ is given by $P_{ \delta}(x,y,z,\varepsilon) = x^{2}y^{k_{1}} - y^{k_{2}}z$, then, after multiplying the vector field by $w^{\delta}$, system \eqref{eq-teo-non-nh-fold} can be rewritten as 
\begin{equation}\label{eq-teo-non-nh-fold-2}
\left\{
  \begin{array}{rcl}
   u' & = & \big{(}\displaystyle\frac{u^{2}w^{\delta}}{w^{\omega_{1}+k_{1}\omega_{2}}} - \frac{vw^{\delta}}{w^{\omega_{3}+k_{2}\omega_{2}-\omega_{1}}}\big{)} + 
   \displaystyle\sum_{d=-1}^{\delta-1}w^{\delta-d}P_{ d} -  \varepsilon u \frac{\omega_{1}}{\omega_{2}}\displaystyle\sum_{d=-1}^{\delta}w^{\delta-d}Q_{ d},  \\
   v' & = & \varepsilon\displaystyle\sum_{d=-1}^{\delta}w^{\delta-d}\big{(}R_{ d} -  v \frac{\omega_{3}}{\omega_{2}}Q_{ d}\big{)},  \\
   w' & = & -\displaystyle\frac{\varepsilon}{\omega_{2}}\sum_{d=-1}^{\delta}w^{\delta+1-d}Q_{ d},
  \end{array}
\right.
\end{equation}
in which the polynomial functions $P,Q$ and $R$ are applied in $(u,1,v,\varepsilon)$. Therefore, setting $w = 0$ and $\varepsilon = 0$ in equation \eqref{eq-teo-non-nh-fold-2}, it follows that the origin of the chart $U_{2}$ is a generic fold singularity if, and only if, $k_{1}\omega_{2} = \delta - \omega_{1}$ and $k_{2}\omega_{2} = \delta + \omega_{1} - \omega_{3}$.\end{proof}

Theorem \ref{teo-normal-forms} gives conditions on the highest quasi homogeneous degree of the polynomial $P$ and on the weights $\omega = (\omega_{1},\omega_{2},\omega_{3})$ in order to assure that the origin of the chart $U_{2}$ is one of the non normally hyperbolic singularities given by Theorem \ref{teo-fenichel-normal-form-fold}. However, it is important to remark that, depending on the weight vector $\omega$, it is not possible to generate such singularities. This fact will be clear in the next examples.

\begin{example}\label{exe-compact-non-NH-fold}
Under the hypothesis of Theorem \ref{teo-normal-forms}, suppose that $\omega_{1} = 1$ and $\omega_{2} = \omega_{3} = 2$. Then for any positive integers $k_{1} = k_{2} = \frac{\delta - 1}{2}$, the origin of the chart $U_{2}$ will be a fold singularity of $X^{\infty}_{ \varepsilon}$. However, if $\omega_{1} = 3, \omega_{2} = 2$ and $\omega_{1} = 1$, then it does not exist positive integers $k_{1}$ and $k_{2}$ satisfying conditions (a) of Theorem \ref{teo-normal-forms}. See Figure \ref{fig-compact-non-NH}.
\end{example}

\begin{example}\label{exe-compact-non-NH-transcritical}
Suppose that $\omega_{1} = \omega_{3}$ and $\omega_{2} = 1$. For any positive integers $k_{1} = k_{2} = \delta - \omega_{1}$, the origin of the chart $U_{2}$ will be a transcritical singularity of $X^{\infty}_{ \varepsilon}$. See Figure \ref{fig-compact-non-NH}.    
\end{example}

\begin{example}\label{exe-compact-non-NH-pitchfork}
Suppose that $\omega_{1} = 2$ and $\omega_{2} = \omega_{3} = 1$. For any positive integers $k_{1} = \delta - 1$ and $k_{2} = \delta - 4$, the origin of the chart $U_{2}$ will be a pitchfork singularity of $X^{\infty}_{ \varepsilon}$. Nevertheless, if $\omega_{1} = 3, \omega_{2} = 2$ and $\omega_{1} = 1$, then it does not exist positive integers $k_{1}$ and $k_{2}$ satisfying conditions (c) of Theorem \ref{teo-normal-forms}. See Figure \ref{fig-compact-non-NH}.    
\end{example}

\begin{figure}[h!]
	\center{\includegraphics[width=0.3\textwidth]{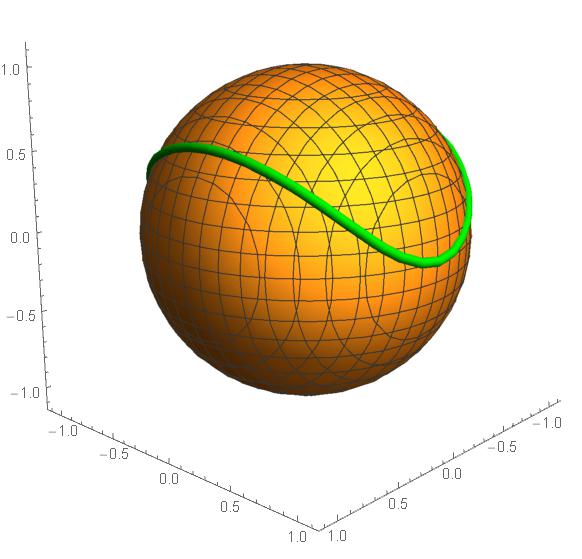}\includegraphics[width=0.3\textwidth]{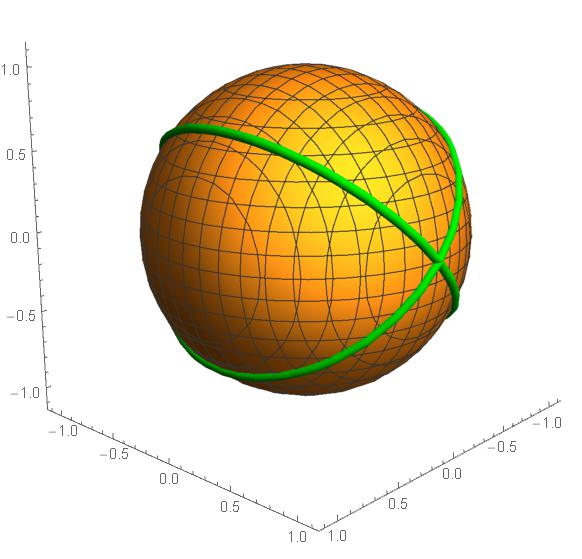}\includegraphics[width=0.3\textwidth]{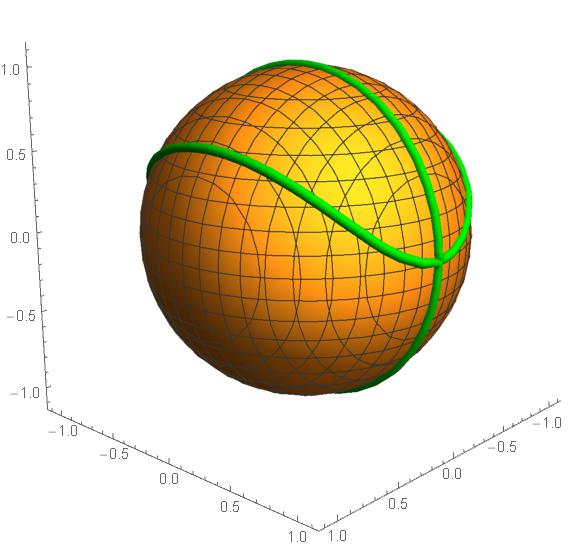}}
	\caption{\footnotesize{Generic non normally hyperbolic singularities at infinity. From the left to the right: fold (Example \ref{exe-compact-non-NH-fold}), transcritical (Example \ref{exe-compact-non-NH-transcritical}) and pitchfork (Example \ref{exe-compact-non-NH-pitchfork}). The critical manifold is highlighted in green.}}
	\label{fig-compact-non-NH}
\end{figure}

\section{Acknowledgements}

This article was possible thanks to the scholarship granted from the Brazilian Federal Agency for Support and Evaluation of Graduate Education (CAPES), in the scope of the Program CAPES-Print, process number 88887.310463/2018-00, International Cooperation Project number 88881.310741/2018-01.

Otavio Perez is supported by Sao Paulo Research Foundation (FAPESP) grant 2021/10198-9. Paulo R. Silva is partially supported by Sao Paulo Research Foundation (FAPESP) grant 2019/10269-3, and CNPq grant 302154/2022-1.

\end{document}